\documentclass[11pt]{amsart}
   \usepackage{amsthm,amsmath,amssymb,amscd,graphicx,enumerate, stmaryrd,xspace,verbatim, epic, eepic,color,url}

\usepackage{eurosym}
\usepackage[all]{xypic}
\SelectTips{cm}{}
\usepackage{pdfsync}
 
   \newtheorem{thm}{Theorem}[subsection]
      \newtheorem*{thm*}{Theorem}
\newtheorem*{thmA}{Theorem A}
\newtheorem*{thmB}{Theorem B}
\newtheorem*{thmC}{Theorem C}
   \newtheorem{proposition} [thm]{Proposition} 
    \newtheorem{fact}[thm]{Fact}
   
      \newtheorem{defilemma}[thm]{Definition/Lemma}
   \newtheorem{prop}[thm] {Proposition}     
   \newtheorem{lemma} [thm]{Lemma}

\usepackage{hyperref}
\usepackage{pgf}

\hypersetup{
    colorlinks=true,       
    linkcolor=blue,          
    citecolor=blue,        
    filecolor=blue,      
    urlcolor=blue           
}

   \newtheorem*{conjecture*}{Conjecture}

\theoremstyle{definition}
 
          \newtheorem*{exercise*}{Exercise}
     \newtheorem{example}[thm]{Example}
 \newtheorem{definition}[thm]{Definition} 
  \newtheorem{defi}[thm] {Definition}

  \newtheorem{remark} [thm]{Remark}

\newcommand{\N}{{\mathbb{N}}}

\newcommand{\PP}{{\mathbb{P}}}
\newcommand{\RR}{{\mathbb{R}}}
\newcommand{\ZZ}{{\mathbb{Z}}}

\newcommand{\cC}{{\mathcal C}}

\newcommand{\cE}{{\mathcal E}}

\renewcommand{\cL}{{\mathcal L}}
\newcommand{\cM}{{\mathcal M}}

\newcommand{\cO}{{\mathcal O}}

\newcommand{\cS}{{\mathcal S}}

\newcommand{\cU}{{\mathcal U}}
\newcommand{\cV}{{\mathcal V}}

\newcommand{\cX}{{\mathcal X}}

 \newcommand{\md}{{\underline{d}}}

  \newcommand{\mk}{{\underline{k}}}

 \newcommand{\mdeg}{{\underline{\deg}}}

\def\<{\langle}
\def\>{\rangle}

\newcommand{\oP}{\overline{P}}

\newcommand{\red}{{\operatorname{red}}}

\newcommand{\Spec}{\operatorname{Spec}}

\newcommand{\Pic}{{\operatorname{{Pic}}}}
\newcommand{\Hom}{{\operatorname{Hom}}}
\newcommand{\Div}{{\operatorname{Div}}}

\newcommand{\Ker}{{\operatorname{Ker}}}
\newcommand{\Aut}{{\operatorname{Aut}}}

\newcommand{\Trop}{{\operatorname{Trop}}}

\newcommand{\trop}{{\operatorname{trop}}}

\newcommand{\an}{{\operatorname{an}}}
\newcommand{\val}{{\operatorname{val}}}
 
   \newcommand{\MON}{\operatorname{D}}
    \newcommand{\EFF}{\operatorname{E}}

\newcommand{\double}{\genfrac..{0pt}1
{\raise -2pt\hbox{$\scriptstyle\longrightarrow$}}{\raise 4pt\hbox
{$\scriptstyle\longrightarrow$}}} 
 
\newcommand{\ssm}{\smallsetminus}

 \newcommand{\la}{\longrightarrow}
\newcommand{\ha}{\hookrightarrow}

\newcommand{\ov}{\overline}

\newcommand{\ma}{\mathcal}

 \newcommand{\Mgbst}{{\ov{\mathcal{M}}_g}}
  \newcommand{\Sgbst}{{\ov{\mathcal{S}}_g}}
    \newcommand{\Sgbsto}{\ov{\mathcal{S}}_g^-}
      \newcommand{\Sgbste}{\ov{\mathcal{S}}_g^+}
 \newcommand{\Sgnbst}{\ov{\mathcal{S}}_{g,n}}
 \newcommand{\Sgnbsto}{\ov{\mathcal{S}}_{g,n}^-}
 \newcommand{\Sgnbste}{\ov{\mathcal{S}}_{g,n}^+}
 \newcommand{\Mgnb}{\ov{M}_{g,n}}
  \newcommand{\Sgnb}{\ov{S}_{g,n}}
  \newcommand{\Mgnbst}{{\ov{\cM}}_{g,n}}
  \newcommand{\Mgnst}{{\mathcal{M}_{g,n}}}

    \newcommand{\Mgnban}{\Mgnb^{\operatorname{an}}}
    \newcommand{\Sgnban}{\Sgnb^{\operatorname{an}}}

\newcommand{\Mt}{M^{\rm{{trop}}}}

\newcommand{\Mgnt}{{M_{g,n}^{\rm trop}}}

\newcommand{\Mgtnb}{\ov{M}_{g,n}^{\trop}}

\newcommand{\Sgnt}{S_{g,n}^{\trop}}
\newcommand{\Sgntb}{\ov{S}_{g,n}^{\trop}}
 
  \newcommand{\Sgntbo}{(\Sgntb)^-}
  
   \newcommand{\Sgntbe}{(\Sgntb)^+}

\newcommand{\Mgntb}{\ov{M}_{g,n}^{\rm trop}}

 \newcommand{\hX}{\widehat{X}}
  \newcommand{\hcX}{\widehat{\mathcal{X}}}
\newcommand{\hY}{\widehat{Y}}
\newcommand{\hf}{\widehat{f}}
\newcommand{\hL}{\widehat{L}}

\newcommand{\hG}{\widehat{G}}
\newcommand{\hH}{\widehat{H}}

\newcommand{\hv}{\widehat{v}}

 \newcommand{\ra}{\rightarrow}
\newcommand{\mb}{\mathbb}
\newcommand{\col}{\colon}
\newcommand{\ol}{\overline}
\newcommand{\wh}{\widehat}
\newcommand{\wt}{\widetilde}

 \newcommand{\Sgn}{{\mathcal{G}}_{g,n}} 
 \newcommand{\Spgn}{\mathcal{SP}_{g,n}}
 \newcommand{\Cgn}{\mathcal{C}_{g,n}}
 
      \newcommand{\tcS}{{\wt{\ma S}}}
  \newcommand{\hcS}{{\widehat{\ma S}}}
    
\begin{document}

\bibliographystyle{plain}

\title{Tropicalizing the moduli space of   spin curves}

\author[]{Lucia Caporaso, Margarida Melo, and Marco Pacini}
\address[Caporaso]{Dipartimento di Matematica e Fisica\\ Universit\`{a} Roma Tre \\ Largo San Leonardo Murialdo \\I-00146 Roma\\  Italy }\email{caporaso@mat.uniroma3.it}
 \address[Melo]{Dipartimento di Matematica e Fisica\\ Universit\`{a} Roma Tre \\ Largo San Leonardo Murialdo \\I-00146 Roma\\  Italy\\ and CMUC and University of Coimbra\\Apartado 3008\\
EC Santa Cruz\\3001--501 Coimbra \\ Portugal}\email{melo@mat.uniroma3.it}
\address[Pacini]{Instituto de Matem\'atica, Universidade Federal Fluminense  \\ Campus do Gragoat\'a \\ 24.210-201 Niter\'oi, Rio de Janeiro, Brazil}\email{pacini.uff@gmail.com, pacini@impa.br}
 
\begin{abstract} 
We study  the tropicalization of    the moduli space of algebraic  spin curves, $\Sgnbst$.
We exhibit its combinatorial stratification and prove that the strata are irreducible.
We construct the moduli space of  tropical spin curves $\Sgntb$, prove that is naturally isomorphic to the skeleton of the  analytification,  $\Sgnban$, of   $\Sgnbst$,  and 
give a geometric interpretation of the retraction of $\Sgnban$ onto its skeleton 
in terms of a tropicalization map $\Sgnban\to \Sgntb$.

 \end{abstract}


\maketitle


 
\tableofcontents

 \section{Introduction and preliminaries}

 \subsection{Introduction}

In recent years the geometry of toroidal compactifications of moduli spaces has been treated more and more from the point of view of tropical geometry. If $\cU\hookrightarrow \ma Y$ is a toroidal embedding of Deligne-Mumford stacks, the boundary, $\ma Y\smallsetminus \ma U$, is endowed with a  stratification whose structure is  encoded in the skeleton, $\Sigma(\ma Y)$, of the Berkovich analytification of $\ma Y$. This skeleton is  a topological retraction of the analytification and   has the structure of a generalized cone complex. On the other hand,  if $\ma Y$ is a moduli space,   $\Sigma(\ma Y)$ has been described, in some remarkable cases, as the moduli space of objects which can be viewed as the tropical counterpart of the objects parametrized by $\ma Y$.
See \cite{ACP15}, \cite{Uli15} \cite{CMR16}, \cite{AP18}, for example.

This paper studies  the moduli space of theta characteristics on smooth curves of genus $g$, written $\cS_g$, and its compactification via stable spin curves.
 The moduli space of  stable spin curves, $\Sgbst$, was   constructed by Cornalba in \cite{cornalba}. It is endowed with a finite (ramified) map  of degree $2^{2g}$, written $\pi:\Sgbst\to \Mgbst$, onto the moduli space of stable curves, $\Mgbst$. The fiber of $\pi$  over a curve $X$   parametrizes  theta characteristics on partial normalizations of $X$. The space  $\Sgbst$ has two connected components,  $\Sgbste$  and $\Sgbsto$, parametrizing  theta characteristics whose  space of global sections has even or odd dimension, respectively.
There has been in the years a constant interest in  the geometry of $\Sgbst$ and its applications; see \cite{J98}, \cite{J00}, \cite{JKV01}, \cite{BF06}, \cite{F10}, \cite{FV14}, for   example. 
Here we pursue the following route:

\begin{enumerate}[(1)]
\item We lay the combinatorial groundwork by defining spin graphs and (extended) spin tropical curves. 
We construct  the moduli space of  tropical spin curves, $\ol{S}_{g}^{\trop}$,
and  the natural map $\pi^{\trop}\col \ol{S}_{g}^{\trop}\ra \ol{M}_{g}^{\trop}$ to the moduli space of tropical curves.
\item 
We describe the stratification of the toroidal embedding $\ma S_g\hookrightarrow \Sgbst$ using spin graphs and  exhibit a recursive presentation of the stata. We prove that the strata are irreducible.
\item We give a modular interpretation of the skeleton of the analytification of $\Sgbst$, by showing that it is isomorphic to $\ol{S}_{g}^{\trop}$ as an extended generalized cone complex. 
\end{enumerate}

Theta characteristics on a   tropical curve, $\Gamma$, of genus $g$  have been  studied in \cite{zharkov}, in case $\Gamma$ is ``pure"  (i.e. such that $g=b_1(\Gamma)$, see \ref{subtrop}).
They  are defined as divisors $D$ such that $2D$ is equivalent to  the canonical divisor of $\Gamma$, and are in bijection with the elements of  $H_1(\Gamma, \mathbb Z/2\mathbb Z)$, so there are  $2^{g}$ of them, of which  exactly one is non-effective.


Theta characteristics on smooth algebraic curves tropicalize to theta characteristics on the skeleton (a tropical curve) and
the tropicalization map has an interesting behaviour.  Indeed, each  effective theta characteristic in the skeleton is the image of $2^g$  theta characteristics, exactly half of which are odd;  the unique non-effective theta characteristic is the image of $2^g$ even theta characteristics;  this  has been  proved  in \cite{lenjensen} generalizing results of  \cite{panizzut} in the case of hyperelliptic curves. 

We briefly  mention that, in the case of plane quartic curves,  an interesting relation has been discovered between classical and tropical bitangents in
   \cite{BLMPR16},  \cite{LM17}, and \cite{CJ17}. 
   
On a different vein, the behaviour of the tropicalization map indicates that a moduli space parametrizing theta characteristics on tropical curves defined as above could not be the skeleton of the moduli space of stable spin curves, as odd and even theta characteristics, which live in different connected components in the algebraic world, may tropicalize to the same object. 
For this reason, 
our notion of a spin structure on a tropical curve is a bit different in that it consists of a theta characteristic enriched by a parity function.

Much less is known about higher spin curves, or   arbitrary roots of divisors,   on the tropical side. 
On the other hand the algebro-geometric counterpart
has  been the object of much interest, and the corresponding
 moduli spaces    have been 
widely studied together with their compactifications;    see for example \cite{J00}, \cite{AJ03},  \cite{CCC07}, \cite{C08}. 
The combinatorial structure of these compactifications is quite intricate and we believe it would be of great interest to study it from the tropical, and non-Archimedean, point of view.

\subsection{Outline of the paper}
In Section~\ref{sec:spingraphs} we introduce the notion of  spin graph, tailored to represent the combinatorial data associated to an algebraic   spin curve. 
A spin graph  is a triple $(G,P,s)$ where $G$ is a stable  graph and $(P,s)$ a spin structure, consisting of a cyclic subgraph $P$ of $G$, and a parity function defined on the vertices of the contracted graph $G/P$. We  work   with graphs with legs; although the legs are   irrelevant to the spin structure, their presence  is essential as it enables us   to use recursive   arguments. 

Then we define a  spin tropical curve as a  tropical curve whose underlying graph is endowed with a spin structure. For every tropical curve $\Gamma$, we denote by $S_\Gamma^{\trop}$ the set of spin structures of $\Gamma$.

There is a natural graded poset structure on the set of spin graphs of  fixed genus, given by the contractions of spin graphs, that is, contractions of graphs which are compatible with the spin structures. As is well known, contractions of graphs correspond to specializations of tropical curves, and in fact we use this  poset  structure to construct a moduli space of   spin tropical curves, $\ol{S}_{g,n}^{\trop}$, as an extended generalized cone complex. The cells of this space parametrize  extended  spin tropical curves with a fixed underlying spin graph.  The following statement summarizes the main results of Section~\ref{sec:spingraphs}.

\begin{thmA}
The extended generalized cone complex $\ol{S}_{g,n}^{\trop}$ has pure dimension $3g-3+n$ and has two connected components corresponding to the parity of the spin structure. There is a natural map $\pi^{\trop}\col \ol{S}_{g,n}^{\trop}\ra \ol{M}^{\trop}_{g,n}$ which is a morphism of extended generalized cone complexes, and for every extended tropical curve $\Gamma$ we have $(\pi^{\trop})^{-1}(\Gamma)\cong S_{\Gamma}^{\trop}/\Aut(\Gamma)$. 
\end{thmA}

In Section~\ref{sec:algspin} we consider the moduli space of stable spin curves. 
As mentioned earlier, we actually consider more generally the moduli space of $n$-pointed stable spin curves $\overline{\mathcal S}_{g,n}$: for $n=0$ this is the space constructed by Cornalba in \cite{cornalba}, while for $n>0$ this is a particular case of the construction of Jarvis in \cite{J00}. A stable spin curve over a stable curve $X$ is given by a theta characteristic on a partial normalization of $X$. So, a stable spin curve has a combinatorial type which is given by the dual graph, $G$, of $X$, the   set, $P$, of  nodes of the partial normalization, and a parity function, $s$, corresponding to the parity of the space of sections on each connected component of the partial normalization. These data are encoded in the dual spin graph
of the stable spin curve. We set $\mathcal S_{(G,P,s)}$ to be the locus in $\overline{\mathcal S}_{g,n}$ parametrizing stable spin curves with combinatorial type $(G,P,s)$. The following statement summarizes informally some  results of Section~\ref{sec:algspin}.

\begin{thmB}
We have a toroidal open embedding ${\mathcal S}_{g,n}\hookrightarrow \overline{\mathcal S}_{g,n}$ which admits a stratification   by loci of the form $\mathcal S_{(G,P,s)}$, and such that the poset of the closures of the strata is governed by the graded poset of spin graphs.
\end{thmB}
We devote Section~\ref{sec:irr} to 
  prove  that the strata $\mathcal S_{(G,P,s)}$ are irreducible.  For our purposes  we  actually need    a stronger statement, Theorem~\ref{thm:irr}, from which the irreducibility of $\mathcal S_{(G,P,s)}$ follows easily.
The proof is by induction,  based on a direct proof  in some special cases. 
The  methods we develop here can  be  applied to  different  geometric  problems. As an application towards Brill-Noether theory we conclude Section~\ref{sec:irr}  by proving  that, on a general stable curve, a theta characteristic (if it exists) has at most a one-dimensional space of sections.

In Section~\ref{sec:tropSg}  we   describe the skeleton of the Berkovich analityfication $\Sgnban $ of $\overline{\mathcal S}_{g,n}$ and   show that there is an isomorphism, $
\ol{\Sigma}(\ol{\ma{S}}_{g,n})\cong \ol{S}_{g,n}^{\trop}$,  between this skeleton and  our moduli space of spin tropical curves. 
Our construction is consistent with the tropicalization map  for   curves
$ \Trop_{\Mgnbst}:\Mgnban  \to \ol{M}_{g,n}^{\trop}$ and 
the  results of the Section~\ref{sec:tropSg}   imply the following
\begin{thmC}
There is a tropicalization map,
$ 
 \Trop_{\Sgnbst}:\Sgnban  \to \ol{S}_{g,n}^{\trop}
 $, 
 fitting in the following commutative diagram
\begin{eqnarray*}
\SelectTips{cm}{11}
\begin{xy} <16pt,0pt>:
\xymatrix{
\Sgnban \ar[d]_{\pi^{\an}} \ar[rr]^{{\Trop}_{\ol{\ma {S}}_{g,n}}\;}   
  &&  \ol{S}_{g,n}^{\trop}   \ar[d]^{\pi^{\trop}} \\
\Mgnban  \ar[rr]^{{\Trop}_{\ol{\ma {M}}_{g,n}}\;}&& \ol{M}_{g,n}^{\trop}  
 }
\end{xy}
\end{eqnarray*}
  \end{thmC}


      \subsection{Graphs}
      \label{subsec:graph}
      Throughout the paper, $G=(V,E, L,w)$ is 
a {\it weighted graph with legs},  or simply a \emph{graph},   defined by the following set of data:
 \begin{enumerate}
 \item
 A non empty finite set, $V=V(G)$, of {\it vertices}, endowed with a  {\it weight}  function   $ w\col V\to \N$.
\item
A finite set $H=H(G)$ of  {\it half-edges} endowed with an involution $\iota:H\to H$, and an endpoint  map
$\epsilon:H\to V$. 
 \item
A set  $E=E(G)$  of \emph{edges} of $G$, defined as  the set of pairs  $e=\{h,h'\}$ of (distinct) elements of $H$ interchanged by $\iota$.  
\item
A set  $L=L(G)$ of \emph{legs} of $G$ defined as the set of   fixed points of $\iota$.

\end{enumerate}
   
If the weight function $w$ is identically zero we say $G$ is  {\it weightless}. 

   We set $n=|L(G)|$;
if $n=0$  we   say $G$ is {\it leg-free}. 
To a  graph $G$ we   associate the leg-free graph 
$ 
  [G]:=G-L 
  $, 
obtained from $G$ by   removing the legs. 

Given $v\in V$, we denote by $\ell(v)=|\epsilon^{-1}(v)\cap L|$ the number of legs ending in $v$, 
and by  $\deg(v)=|\epsilon^{-1}(v)\smallsetminus L|$   the number of half-edges other than legs ending in $v$. 

 The {\it genus} of $G$ is  
$
g(G):=\sum_{v\in V} w(v) +b_1(G),
$
where $b_1(G)=|E| - |V| +c(G)$, and $c(G)$ is  the number of connected components of $G$. 
Obviously, $g(G)=g([G])$. 

If $F\subset E $, we write   $\langle F\rangle$  for the leg-free subgraph of $G$ spanned by $F$. We will  abuse notation and write
$b_1(F)=b_1( \langle F\rangle)$.
Next,   $G-F$ denotes the graph obtained by removing $F$, so that $G$ and $G-F$ have the same vertices and the same legs.
We denote by $G-F^o$ the graph obtained by replacing every edge $e=\{h,h'\}$ in $F$ by a pair of legs
having   the same endpoint as $h$ and $h'$. Hence 
$[G-F^o]=[G-F]$ and 
$$
|L(G-F^o)|=|L(G)|+2|F|.
$$

The group of divisors on $G$, written $\Div(G)$, is the free abelian group generated by $V$.
We write $\md=(\md_v, \; v\in V)$ for a divisor of $G$. The {\it degree} of a divisor $\md$ is the integer $|\md|:=\sum_{v\in V}\md_v$. 
The \emph{canonical divisor} of $G$, written $\underline{k}_G$, is defined, for all $v\in V$,  as follows
$$(\underline{k}_G)_v=2w(v)-2+\deg (v).$$
Notice that $\underline{k}_G$  has degree $2g(G)-2c(G)$.

We say that $G$ is {\it stable} (respectively, {\it semistable}) if  it is connected and if  for every vertex $v\in V(G)$ we have
\[
2w(v)-2+\deg(v)+\ell(v)> 0 \quad (\text{respectively}, \geq 0).
\] 

The isomorphism classes of stable graphs of genus $g$ with $n$ legs form a finite set denoted by $\Sgn$.  We will use  the following two well known facts:
 \begin{enumerate}
 \item
 $\Sgn$ is not empty if and only if $2g-2+n>0$.  
 \item
If $G$ is stable and $F\subset E(G)$, then every connected component of $G-F^o$ is stable.
\end{enumerate}

If $G$ is weightless and $\deg(v)+\ell(v)= 3$ for every vertex $v\in V(G)$, we say that $G$ is \emph{$3$-regular.} It is easy to check that a $3$-regular graph of genus $g$ with $n$ legs has $3g-3+n$ edges.

  \subsection{Cycles and Contractions}\label{sec:contractions}
 
Let $G$ be a graph. Following \cite{D97}, we denote by $\cE_G$ (respectively, $\cV_G$) the vector space over $\Bbb {F}_2$ spanned by $E$ (respectively, by $V$).
Every element of $\cE_G$ can be identified with a subset, $F$,  of $E$, and we shall abuse notation by using the same symbol, $F$,  for the vector    and the subgraph, $\langle F\rangle$. If $G$ has legs, then the subgraph spanned by $E$ is $[G]$,
but we shall abuse notation again and denote by $G\in \cE_G$ (instead of $[G]$) the element corresponding to $E$. 

We have a linear map $\partial\col \cE_G\to \cV_G$ such that $\partial(e)=u+v$ where $u$ and $v$ are the ends of $e$.
The kernel of $\partial$ is called the  {\it cycle space} of $G$,  and denoted  
$
\cC_G:=\Ker \partial.
$

It is well known that $|\cC_G|=2^{b_1(G)}$, and  that    $F\in \cE_G$ lies in $\cC_G$ if and only if the graph  $\langle F\rangle$ has no vertex of odd degree 
(such graphs are called {\it Eulerian}  if  they are connected). We refer to elements of $\cC_G$, and to the subgraphs of $G$ they span, as {\it cyclic}. 
If $F\in \cC_G$ and $b_1(F)=1$   we say   that $F$ is a {\it cycle}.

   A  {\it {morphism}} $\eta\col G\to G'$  between two graphs   is given by 
   a map  
   $\eta\col V(G)\cup H(G)\to V(G')\cup H(G')$, compatible with the graph structure. That is,
writing $\epsilon'$ and $\iota'$
   for the endpoint map and the involution of $G'$,
   for every $x\in V(G)\cup H(G)$ we have 
   $$\eta\circ(id_{V(G)}\cup\epsilon)(x)=(id_{V(G')}\cup\epsilon')\circ \eta(x)
   $$
   $$ \eta\circ(id_{V(G)}\cup\iota)(x)=(id_{V(G')}\cup\iota')\circ \eta(x). 
   $$
Hence  
  $\eta$ induces   maps
$\eta^V \col V(G)\to V(G')$, $\eta^E\col E(G)\to E(G')\cup V(G')$ and    $\eta^L\col L(G)\to L(G')\cup V(G')$.
 An   {\it {isomorphism}} is a morphism $\eta$  such that
$\eta^V$ is a weight-preserving  bijection,   $\eta^E$   is a bijection
between  $E(G)$ and $ E(G')$, and   $\eta^L$ an order-preserving bijection between  $L(G)$ and $L(G')$.

We write $\Aut(G)$ for the group of automorphisms of $G$  (i.e., isomorphisms of $G$ with itself). Notice that an automorphism of $G$ is the identity on $L$.

Given   a subset $F\subset E$, the {\it contraction}  of $F$ is the graph $G/F$ together with   the morphism 
  $$
\gamma\col G\la  G/F 
$$ 
given by   contracting all edges $e\in F$. The legs of $G/F$ are 
identified with the legs of $G$, and the edges of $G/F$ are identified with $E(G)\smallsetminus F$.
The weight function $ V(G/F)\to \N$
maps $v\in V(G/F)$ to $g(\gamma^{-1}(v))$.
 
Notice that  $G$ is connected if and only if  so is $G/F$. We   have 
  $$
  b_1(G/F)=b_1(G)-b_1(F),
  $$
 and  the genus of $G$ equals the genus of $G/F$. 
   
   A contraction $\gamma\col G\to H$ induces a   homomorphism
   $ 
  \gamma_*\col \Div(G)\to \Div(H)  $ 
   such that for $\md\in \Div(G)$ and $v\in V(H)$ we have $(\gamma_*\md)_v=\sum_{u\in \gamma^{-1}(v)}\md_u$.

 \subsection{Algebraic curves}\label{sec:algcur}

In this paper, $k$ denotes an algebraically closed field of characteristic not equal to $2$,
and $g,n$   non-negative integers. 

A \emph{curve} is a reduced, projective variety of dimension one,
not necessarily connected, over $k$.
We  assume that our curves have at most nodes as singularities. The genus of a curve is the arithmetic genus. We write $\omega_X$ for the dualizing sheaf of a curve $X$.

A {\it pointed curve} is a pair $(X,\sigma)$, where $X$ is a  curve and $\sigma$ is an ordered finite set of smooth and distinct points of $X$.   

As usual, the {\it dual graph} $G= (V,E,L,w)$ of an  $n$-pointed curve $(X,\sigma)$ is defined so that $V$ is the set of irreducible components of $X$, the weight of a vertex   is the genus of the normalization of the corresponding component,
  $E$ is the set of nodes of $X$,  and to every point in $\sigma$ there corresponds a leg  ending at the vertex
corresponding to the component to which the point belongs.
For $v\in V$ we write $C_v\subset X$ for the corresponding   component, while we use the same notation for edges (resp. legs) of $G$ and nodes (resp.    points) of $X$.
The genus of $G$ is equal to the genus of $X$.

We define $(X,\sigma)$  to be {\it stable}, or {\it semistable}, if  so is its dual graph. 

Let $X$ be a nodal curve and  $E\subset  X$  a smooth rational component; we say $E $ is {\it exceptional} if $|E\cap \ov{X\smallsetminus E}|=2$ and $|\sigma\cap E|=0$. A node of $X$ is \emph{exceptional} if it lies on an exceptional component of $X$. A stable curve has no exceptional component.

A pointed curve   is   {\it quasistable} if  it is semistable and if 
    two exceptional components of $X$ do not intersect.

Let $(\wh X,\sigma)$ be a quasistable   curve. There is a stable   curve, written $st(\wh X,\sigma)$, obtained by contracting all the exceptional components of $\wh X$; we have a unique contraction morphism $\wh X\ra st(\wh X)$.  Conversely, given a stable   curve $(X, \sigma)$ and a set of nodes $R$ of $X$, we   construct the quasistable   curve $(\wh X_R,\sigma)$, called the {\it blow up} of $X$ at $R$, such that $X=st(\wh X_R)$, where the set of exceptional components of $\wh X_R$ is contracted to $R$. We write $X^{\nu}_R$ for the normalization of $X$ at $R$, so that
\begin{equation}\label{eq:XtildeR}
\hX_R = X^{\nu}_R\cup (\cup_{r\in R} E_r)
\end{equation}
with $E_r$ an exceptional component.
We shall view $ X^{\nu}_R$ as a pointed curve,  $(X^{\nu}_R, \sigma_R)$, 
as follows. Let $\nu_R\col X^{\nu}_R\to X$ be the normalization map,
then
$\sigma_R=\sigma\cup \nu_R^{-1}(R)$. So, the dual graph of $(X^{\nu}_R, \sigma_R)$ is $G-R^o$.
Notice that every connected component  of $(X^{\nu}_R, \sigma_R)$ is stable, as a pointed curve.

The  dual graph of $\hX_R$ will be denoted by $\hG_R$. Clearly $\hG_R$ is obtained from $G$ by inserting a vertex,
$\hv_r$, in the interior of every edge in $R$.
We have a natural inclusion $V(G)\subset V(\widehat{G}_R)$ and we refer to the vertices of type $\hv_r$ as the {\it exceptional vertices} of $\widehat{G}_R$.  
Summarizing, we have
\[
V(\hG_R)=V(G)\cup \{\hv_r,\;  \forall r\in R\}.
\]

Let $\Pic(X)$ be the Picard scheme of a curve $X$, parametrizing line bundles on $X$ up to isomorphism. There is a    homomorphism, $\mdeg\col \Pic(X)\to \Div(G)$ sending a line bundle to its multidegree.
We have
$ 
\mdeg\  \omega_X=\underline{k}_G.
$ 

A {\it theta characteristic}  is a line bundle $L$ on $X$ such that $L^2\cong \omega_X$. 
We set
\begin{equation}
 \label{SZ0}
{\cS}_{X}^0:=\{L\in \Pic(X):\  L^2\cong \omega_{X}\}.
\end{equation}
We have
\begin{equation}
 \label{degCv}
\deg_{C_v} \omega_{X}=2w(v)-2 +\deg(v)
\end{equation} for every vertex $v$ of $G$. Hence
 a theta characteristic  on $X$ exists if and only if $\deg(v)$ is  even for every   $v$, if and only if $G$ is cyclic.
 
 Recall that the parity of a theta characteristic, $L$, is the parity of $h^0(X,L)$, and it is  deformation invariant. We write ${\cS}_{X}^{0,+}$ for the  even  theta characteristics  and ${\cS}_{X}^{0,-}$ for the odd ones, so that
$ 
{\cS}_{X}^0={\cS}_{X}^{0,+}\sqcup {\cS}_{X}^{0,-}.
$

The moduli stack of stable $n$-pointed curves of genus $g$ is denoted by $\Mgnbst$. We let $\Mgnst$ be the substack parametrizing smooth    curves.   
The map  
$ 
\Mgnbst \la \Sgn 
$, 
sending  a stable pointed curve to its dual graph, is a surjection to which we will return  later.
\

From now on, $X=(X,\sigma)$ will be  a stable curve of genus $g$ with $n$ marked points, and $G$ its dual graph. We shall always assume that     $2g-2+n>0$.

 \section{Spin graphs and spin tropical curves}\label{sec:spingraphs}
\label{sec:spingraphs}
   \subsection{Spin graphs}
We want to define spin structures on a graph similarly to what is done for curves. As we shall see, a spin structure on a curve $X$ is essentially a theta characteristic on a partial normalization of $X$. Since partial normalizations of $X$ correspond to spanning subgraphs of $G$,  we will define a spin structure on a graph to be a theta characteristic on a spanning subgraph, plus  some extra structure.

Let $G$ be a graph and $\mk_G$ its canonical divisor. If every vertex of $G$ has even degree,  with \eqref{degCv} in mind we   consider 
 the divisor
 $$
 (\mk_G/2)_v=w(v)-1+\deg(v)/2,\quad  \forall v\in V.
 $$
 This divisor  makes sense if and only if $G$ is cyclic, and it plays the role of a theta characteristic on $G$. 
Using the notation introduced in Subsection~\ref{sec:contractions},   we introduce   the following  definition.
\begin{defi}\label{def:spingraph}
 A {\it spin graph} is a triple $(G,P,s)$ such that $G$ is a graph, $P\in \cC_G$ and $s \col V(G/P)\to  \ZZ/2\ZZ=\{0,1\}$ is a  {\it sign} (or {\it parity}) function such that $s(v)=0$ for every vertex $v$ of $G/P$ of weight $0$.  We call $(P,s)$   a {\it spin structure} on $G$. 
 
 The {\it parity} of a spin graph $(G,P,s)$ is the parity of $\sum_{v\in V(G/P)}s(v)$.
\end{defi}

  Recall that  $G$ is assumed to have $n$ legs. Yet,
a spin structure on $G$  is the same as a spin structure on the leg-free graph $[G]$.

 We denote by $SP_G$ the set of spin structures on $G$, and we write
 $$
 SP_G=SP_G^+\sqcup SP_G^-
 $$
 where $SP_G^+$ (resp. $SP_G^-$) denotes the set of even (resp. odd) spin structures.

If the sign function is identically zero, we denote it by $s_0$. The trivial spin structure, $(0,s_0)$, is of course even, hence $SP^+_G\neq \emptyset$.

Observe that 
 $g(G)=0$ if and only if $SP_G=SP_G^+$, if and only if $|SP_G|=1$.

If $G$ has one vertex and no edges, a spin graph on it is of type $(G,0,s)$; in this case we     denote such a spin graph also by $(G,G,s)$.

We have a   forgetful  map
$ 
\phi\col SP_G\la \cC_G
$ mapping $(P,s)$ to $P$.
Set
 $$
SP_{(G,P)}=\phi^{-1}(P), \quad   SP^+_{(G,P)} =SP_{(G,P)}\cap SP_G^+ ,  \quad SP^-_{(G,P)} =SP_{(G,P)}\cap SP_G^-.
$$

\begin{remark}
 Let $G $ be a graph and   $G_0$ the underlying weightless graph.
 Then  $\cC_G=\cC_{G_0}$ and  
 $SP_{G_0}\subset SP_G$,   preserving the parity.
  \end{remark}

For $P\in \cC_G$ 
  we set $R:=E\smallsetminus P$.
With the notation   in \ref{subsec:graph}, define
\[
\ov P:=G- R^o  
\]
so that $\ov{P}$ has the same vertices as  $G$ and is endowed with $(n+2|R|)$ legs. 

Notice that $\ov P$ is connected if and only if $P$ is a spanning subgraph of $G$, if and only if  $|V(G/P)|=1$.
More precisely, the connected components of $\ov P$ correspond to the vertices of $G/P$, so that    we write
 $\{\ov P_v, \  v \in V(G/P)\}$ for  the set of  connected components of $\ov P$.
  
\begin{defi}
 \label{defccspin}
Let  $(G,P,s)$ be a spin graph. For any $v\in V(G/P)$ we define 
  the spin structure $(\ov P_v, s_v)$ on $\ov P_v$ such that $s_v(v):=s(v)$ for $v \in V(G/P)$ (we identify the   vertex of $\ov P_v/{\ov P_v}$ with $v\in V(G/P)$). 

We call $(\ov P_v,s_v)$, for $v\in V(G/P)$, the \emph{connected components} of $(G,P,s)$. 
 \end{defi}

Connecting with the initial observation, for any spin structure $(P,s)$ on $G$, we define the divisor
$ 
\md^P\in \Div(\ov P)
$ as follows
\begin{equation}\label{E:tcg}
\md^P_v=w(v)-1+\deg_{\oP}(v)/2,
\end{equation}
where $\deg_{\oP}(v)$ is the degree of $v$ as a vertex of ${\oP}$. Hence $2\md^P=\mk_{\ov P}$
so that  $\md^P$ 
will be viewed as the theta characteristic of  $\ov{P}$.

We denote by $c^+(\ol P)$ the number of connected components of $\ol P$ of positive genus, i.e. the number of vertices of $G/P$ of positive weight. 

\begin{lemma}
\label{SGP} Let  $P\in \cC_G$. Then
$ 
|SP_{(G,P)}| = 2^{c^+(\ol P)}.
$ 
Moreover,
$$
\quad |SP^+_{(G,P)}|=|SP^-_{(G,P)}|= 2^{c^+(\ol P)-1},
 $$
unless $P=0$ and $G$ is weightless, in which case $|S^+_{(G,0)}|=1$ and  $S^-_{(G,0)}=\emptyset$.
 In particular
 $$
 |SP_G|= \sum_{P\in \cC_G}2^{c^+(\ol P)}\geq  2^{b_1(G)+1}-1,
 $$
 with equality if and only if  $G$ is weightless and $c^+(\ol P)=1$ for every $P\neq 0$.
\end{lemma}
\begin{proof}
 A vertex of $G/P$ has weight nonzero if and only if it is 
obtained contracting a connected component of $P$ of positive genus. The statement follows trivially from this.
\end{proof}

\subsection{Contractions of spin graphs}
Let $\gamma\col G\to G'$ be a morphism of graphs.  
 The maps  $\gamma^V$ and $\gamma^E$ (see Subsection \ref{sec:contractions}) induce  two  linear maps
$$
 \gamma^V_*\col\cV_G\la \cV_{G'} \quad \quad \text{and} \quad \quad \gamma^E_*\col\cE_G\la \cE_{G'},
$$
where $\gamma^E_*(e)=e$  if  $ \gamma^E(e)\in E(G')$ and  $\gamma^E_*(e)=0$ otherwise.

The following easy   lemmas   give some useful functorial properties of a contraction $\gamma\col G\ra G'$ with respect to the maps $\gamma_*^V$ and $\gamma_*^E$. The same statements are easily seen to hold  also when $\gamma$ is an isomorphism.

\begin{lemma}
\label{funct}
Let $\gamma\col G\to G'=G/F$ be a contraction. 
\begin{enumerate}[(a)] \item
The following diagram of linear maps is commutative
 $$
\xymatrix@=.4pc{
&&\cE_G  \ar[dd]_{\partial}    \ar[rrrr]^{\gamma^E_*}  &&&&\cE_{G'} \ar[dd]^{\partial}  &&&   &&&&  \\
\\
 &&\cV_G \ar[rrrr]^{\gamma^V_*} &&&&\cV_{G'}
}
$$
\item
 The restriction of $\gamma^E_*$ to $\cC_G$ induces a surjection, denoted as follows $$\gamma_*\col \cC_G\la \cC_{G'}.$$
  \item
  If $\widetilde{\gamma}\col G'\to G''$ is a contraction, then $(\widetilde{\gamma}  \gamma)_*=\widetilde{\gamma}_*  \gamma_*$.
 \item
 Let $P\in \cC_G$ and set $P' :=\gamma_*P$. The following diagram is commutative
 $$
\xymatrix@=.4pc{
&&G  \ar[dd]    \ar[rrrr]^{\gamma}  &&&& G'  \ar[dd] &&&   &&&&  \\
\\
&&G/P \ar[rrrr] ^{\ov{\gamma}}&&&&G'/P'
}
$$
where $\ov{\gamma}$ is a contraction and $G'/P'=G/(F\cup P)$.
\end{enumerate}
 \end{lemma}
\begin{proof}
 The first three parts are clear. 
  By definition, $P'=P\smallsetminus F$. Hence 
 $ 
 G'/P'=(G/F)/ P' =(G/F)/( P \smallsetminus F)=G/(F\cup P)
 $ 
 \end{proof}
  
  \begin{lemma}
\label{ontoS}
Let $\gamma\col G\to G'$ be a contraction.
\begin{enumerate}[(a)]
 \item
 \label{ontoSa}
There  is  a parity preserving  map (abusing notation)
 $$
 \gamma_* \col SP_G\to SP_{G'};\quad \quad  (P,s)\mapsto (P', s' )
 $$  
 where  $P'=\gamma_*P$ and 
 $s' $
 maps  $v'\in V(G'/P')$ to $\sum_{v\in  {\ov{\gamma}}^{-1}(v')}s(v)$.  
 \item
  \label{ontoSb}
  If $\widetilde{\gamma}\col G'\to G''$ is a contraction, then $(\widetilde{\gamma}  \gamma)_* =\widetilde{\gamma}_*   \gamma_* $.
\end{enumerate}
\end{lemma}

\begin{proof} 
$\gamma_*P$ is cyclic by Lemma \ref{funct} . Hence to show that the image of $ \gamma_*$ is in $SP_{G'}$ it suffices
 to   check  that   $s'(v')=0$
for every vertex $v'$ in $G'/P'$  of weight zero.
Since $\ov{\gamma}\col G/P\to G'/P'$ is a contraction, 
 if $v'$ has weight zero then every vertex $v$ of $G/P$ 
 mapping to $v'$ has weight zero, hence $s(v)=0$
 and hence $s'(v')=0$. 

The fact that $ \gamma_*$ preserves the parity, and 
part \eqref{ontoSb}, are clear.
 \end{proof}

\begin{defi}
\label{defcon}
A \emph{contraction} (resp. an {\it isomorphism})  from a spin graph $(G,P,s)$ to a spin graph $(G',P',s')$ is a contraction (resp. an isomorphism) of graphs $\gamma\col G\to G'$  such that $\gamma_*(P,s)=(P',s')$. 

 We write  $\gamma\col(G,P,s)\to (G',P',s')$.
\end{defi}

Recall that  $\Sgn$ is the set of isomorphism classes of stable graphs of genus $g$ with $n$ legs. 
The contraction of a graph in $\Sgn$ is again in $\Sgn$, and 
 $\Sgn$ is a  poset with respect to the following partial order
\[
 G \geq G',   \text { if }  G'=G/F    \  \text{ for some } F\subset E(G)
\]
(i.e.   if there is a contraction $G\ra G'$).
It is well known that $\Sgn$ is graded, and a rank function on it  is $G\mapsto |E(G)|$.
Moreover,
$\Sgn$  is connected as a poset
(i.e. the associated graph is connected).

\begin{remark}
\label{Gg}
 A  way to see that $\Sgn$ is connected,   useful for later,
 is to consider the graph $G_{g,n}$ made of a single vertex of weight $g$ with $n$ legs.
Then for every $G\in \Sgn$ we have $G\geq G_{g,n}$, and hence $\Sgn$ is connected.
\end{remark}
 
 \subsection{Posets of spin graphs of genus $g$}
 The poset of cyclic graphs of genus $g$ with $n$ legs is the set
$$
 \cC_{g,n}:=\bigsqcup_{G\in \Sgn}\cC_G
$$
partially ordered as follows: $(G,P)\geq (G',P')$ if there exists a contraction $\gamma\col G\to G'$ such that $\gamma_*P=P'$.

Next, the poset of stable spin graphs of genus $g$ is the set 
\[
\Spgn:=\bigsqcup_{G\in \Sgn}SP_G,
\]
endowed with the following partial order:
 $(G,P,s)\geq (G',P',s')$
if there exists a contraction $\gamma\col G\to G'$ such that $\gamma_*(P,s)=(P',s')$.
  
  It is easy to see that the above are partial orders.

\begin{prop}
\label{pospin}
\begin{enumerate}[(a)]
 \item
 The poset $\Cgn$ is connected and the  forgetful map $\Cgn\to \Sgn$   is a quotient of posets.
 \item
 \label{pospinb}
 If $g>0$, the poset $\Spgn$ has two connected components, $\Spgn^+$ and $\Spgn^-$ corresponding to, respectively,
 even and odd spin graphs.
 
  If $g=0$, then  $\Spgn$ is connected and $\Spgn=\Spgn^+$. 
 \item
\label{pospinc}
 The forgetful maps below  are  quotients of posets, the first and the second one for every $g$, while the third one for $g>0$,
$$
\Spgn^+\la \Cgn,\quad \Spgn^+\la \Sgn,  \quad \Spgn^-\la \Sgn.
$$
\item
\label{pospind}
The map sending $(G,P)$ to $|E|$ (respectively, $(G,P,s)$ to $|E|$)  is a rank on $\Cgn$ (respectively, on $\Spgn$).
\end{enumerate}
\end{prop}
\begin{proof}
 Let $G_{g,n}\in \Sgn$ be the graph introduced in Remark \ref{Gg}.
We have
 $\cC_{G_{g,n}}=\{0\}$; moreover, if $g>0$,    we have $SP_{G_{g,n}}=\{(0,s_0), (0,s_1)\}$  (so that $(0,s_0)$ is even and $ (0,s_1)$ is odd), and if $g=0$,   we have     $SP_{G_{g,n}}=\{(0,s_0)\}$. 
  
 Let $G\in \Sgn$ and let $\gamma\col G\to G_{g,n}=G/G$.
 For every $P\in \cC_G$ we have $\gamma_*P=0$, hence for every
 $(G,P)\in \Cgn$ we have $(G,P)\geq  (G_{g,n},0)$. This implies that $\Cgn$ is connected.
 To prove that the forgetful map is a quotient it suffices to observe that if $G\geq G'$
 then $(G,0)\geq (G',0)$.
 
 Let now $(G,P,s)\in \Spgn$. By definition
 $$
\gamma_*(P,s)=    \begin{cases}  (0,s_0),   & \text{ if } (P,s) \text{ is even}; \\
 (0,s_1),   & \text{ if } (P,s) \text{ is odd}.\
 \end{cases}
$$
Therefore every even spin graph $(G,P,s)$ satisfies $(G,P,s) \geq(G_{g,n},0,s_0)$, hence $\Spgn^+$ is connected;
 similarly for  $\Spgn^-$, when non-empty.

Let us show show that if $g>0$, then $SP^-_G$ is not empty for every $G\in \Sgn$. This is clear if $G$ is not weightless.
If $G$ is weightless then $b_1(G)=g>0$,  hence there exists $P\in \cC_G$ such that $b_1(P)>0$, hence
$G/P$ is not weightless. Therefore  $SP_{(G,P)}^-$ is not empty.

It is clear that  the forgetful maps in \eqref{pospinc} are quotients. 

We know that the map $G\mapsto |E|$ is a rank on $\Sgn$.
The partial order on the fibers of $\Cgn\to \Sgn$ is trivial, hence the rank on $\Sgn$ lifts to a rank on $\Cgn$, proving \eqref{pospind} for $\Cgn$. The proof for $\Spgn$ is the same.
\end{proof}
Consider now $\Aut(G)$,  the automorphism group of   $G$. 
By Definition~\ref{defcon}, if $\alpha \in \Aut(G)$ and $(P,s)\in SP_G$ then $\alpha$ maps $P$ to $\alpha_*P$ and $(P,s)$ to $\alpha_*(P,s)$,
which is also an element of $SP_G$.
We set
\[
 \Aut(G,P,s):=\{\alpha\in \Aut(G):\  \alpha_*(P,s)=(P,s)\}.
\]
We need to take into account the  action of $\Aut(G)$ on $\cC_G$ and on  $SP_G$. Set
\begin{equation}
 \label{[Spgn]}
  [\Cgn]:=\bigsqcup_{G\in \Sgn}\cC_G/\Aut(G),\quad  \quad [\Spgn]:=\bigsqcup_{G\in \Sgn}SP_G/\Aut(G).
\end{equation}
 Recall that an automorphism of $G$   preserves the weights.
 Hence $\Aut(G)$   leaves the sets $SP^+_G$ and $SP^-_G$ invariant and we write
 $[\Spgn^+]$ and $[\Spgn^-]$ so that  $ [\Spgn]=[\Spgn^+]\sqcup[\Spgn^-]$.
 
 In what follows the isomorphism class of an object will be denoted using square brackets. For example we write $[P,s]\in SP_G/\Aut(G)$ for the class of $(P,s)$ and $[G,P,s]\in [\Spgn]$ for the class of $(G,P,s)$.
 
 It is clear that the partial orders defined on $\Cgn$ and on $\Spgn$ descend on   $[\Cgn]$ and $[\Spgn]$ (since two spin graphs in the same class are not comparable).
So we shall use the notation
 $$[G,P,s]\geq [G',P',s']$$
  to mean that
 there exist $(G,P,s)\in [G,P,s]$ and $(G,P',s')\in [G,P',s']$ such that $(G,P,s)\geq (G',P',s')$.
Similarly for $[\Cgn]$.
 The forgetful maps below are, again, quotients of posets, the first three for every $g$, the last  for $g>0$
 $$
 [\Cgn]\to \Sgn, \quad [\Spgn^+]\to [\Cgn], \quad [\Spgn^+]\to \Sgn,    \quad [\Spgn^-]\to \Sgn.
$$
\begin{remark}
\label{pospinrk}
 Proposition~\ref{pospin} holds if we replace $\Cgn$  and $\Spgn$  by  $[\Cgn]$ and $[\Spgn]$  (mutatis mutandis). 
\end{remark}
 
 \subsection{Tropical curves and their spin structures}
 \label{subtrop}
We recall the definition of a tropical curve, and some properties of the associated moduli spaces; see \cite{BMV} and \cite{ACP15} for details and references.

A  {\it ($n$-pointed) tropical curve} is a pair $\Gamma=(G,\ell)$ where $G$  is a graph (with $n$ legs) and $\ell\col E\to \RR_{\geq 0}$ is a length function.
 We need to consider the more general case of an {\it extended tropical curve},
 which is defined as above with the difference that $\ell\col E\to \RR_{\geq 0}\cup\{\infty\}$. To simplify the notation we set
 $\ov{\RR}_{\geq 0}=\RR_{\geq 0}\cup\{\infty\}$, endowed with the compact one-point topology. 
 We say that $\Gamma$ is {\it stable} if so is $G$ and that $\Gamma$ is {\it pure} if $G$ is weightless. The \emph{genus} of $\Gamma$ is the genus of $G$.

 The moduli space   of stable $n$-pointed tropical curves of genus $g$ is denoted by $\Mgnt$, and the moduli space of extended ones by $\Mgtnb$.
The first is  a generalized cone complex, the second an extended generalized cone complex.
 They have both dimension $3g-3+n$ and $\Mgnt$ is   open  and dense in $\Mgntb$.
For any   $G\in \Sgn$ we write
 $$
 \Mt_G:=\{[\Gamma] \in \Mgnt:   \Gamma=(G,\ell)\},\quad   \ov{M}_G^{\trop}:=\{[\Gamma] \in \Mgntb: \Gamma=(G,\ell)\}.
 $$

 The spaces $ \Mgnt$ and $ \Mgntb$ are constructed as colimits of suitable diagrams of cone complexes built from the poset $\Sgn$. If for every graph $G$ we denote by $\sigma_G=\mathbb R_{\ge0}^{E}$, $\ov{\sigma}_G=\ov{\mathbb R}_{\ge 0}$, and by $\sigma^o_G$,  $\ov{\sigma}^o_G$ their interiors, then we have  {\it graded stratifications}
 $$
  \Mgnt=\bigsqcup_{G\in \Sgn}  \Mt_G\cong\bigsqcup_{G\in\Sgn} \sigma^o_G/\Aut(G),
$$
$$
 \Mgntb=\bigsqcup_{G\in \Sgn}  \ov{M}_G^{\trop}\cong\bigsqcup_{G\in \Sgn}\ov{\sigma}^o_G/\Aut(G).
 $$
The terminology ``graded stratification" means the folowing. The closure  of a stratum  is a union of strata, and the stratum corresponding to $G$
 contains the stratum  corresponing to $G'$ if and only if $G\geq G'$.
 Moreover, the map associating to each stratum its dimension is a rank on the poset of strata, i.e. on $\Sgn$; see \cite[Def. 1.3.2]{CC18}.

Before defining spin tropical curves, let us take a brief detour. A divisor  on a tropical curve $\Gamma$ is a formal sum $D=\sum_{p\in \Gamma}D(p)p$, where $D(p)\in\mathbb Z$ is non-zero only for a finite number of points $p\in \Gamma$.  A  rational function  on $\Gamma$ is a continuous piecewise-linear function on $\Gamma$ with integer slopes; a  principal divisor on $\Gamma$ is the divisor associated to a principal function. Two divisors are  equivalent  if their difference is a principal divisor. The  canonical divisor,  $K_\Gamma$, of $\Gamma$ is the canonical divisor of the underlying graph $G$, seen as a divisor on $\Gamma$.
A \emph{theta characteristic} on $\Gamma$ is a divisor $D$ such that $2D-K_\Gamma$ is equivalent to a principal divisor on $\Gamma$. 
Theta characteristics on pure tropical curves have been studied by Zharkov, and 
they correspond to elements of the cycle space $\cC_G$ of $G$
(see \cite[Thm. 7]{zharkov}), so a pure tropical curve of genus $g$ admits exactly $2^{b_1(G)}=2^g$ theta characteristics.

Let us define the theta characteristics  on any   tropical curve of genus $g$.
Given an edge $e$ of $G$, we   denote by $p_e\in \Gamma$ the mid-point of $e$. 
For $P\in\cC_G$, let   $D^P$ be the following divisor  
\[
D^P(x)=
\begin{cases}
\underline d^P_x, \text{ if } x \in V;\\
1, \text{ if } x=p_e \text{ for }e\in E\smallsetminus P;\\
0, \text{ otherwise }
\end{cases}
\]
where $\md^P$ is the divisor defined in \eqref{E:tcg}.
It is easy to check that $D^P$ is a theta characteristic on $\Gamma$, and that there are $2^{b_1(G)}$ of them.
Moreover, if $w=0$, one can show that $D^P$ is equivalent to the divisor $\ma K_P$ 
described  in \cite{zharkov}.
In particular, we have

\begin{remark}
The divisors $D^P$ associated to the elements $P$ in $\cC_G$ are (non-equivalent) representatives for the $2^{b_1(G)}$ theta characteristics of $\Gamma$.
\end{remark}

We will not include the proof  as it is not necessary here. 
Our 
  definition of  spin tropical curves  can be seen as an enriched version of   tropical theta characteristics.

\begin{defi}
 A  {\it   (extended) spin tropical curve} 
 is a triple $\Psi=(\Gamma, P,s)$ where $\Gamma=(G,\ell)$ is a (extended)  tropical curve
 and $(P,s)$
  is a spin structure on $G$.  
 The spin tropical curve $\Psi$  is \emph{stable} if so is $G$. 
  
 The genus of    $\Psi$ is the genus of $\Gamma$, and  its  parity is the parity of $(P,s)$.
  \end{defi}
 
Let $\Psi=(\Gamma, P,s)$ and $\Psi'=(\Gamma', P',s')$ be two spin curves, with $\Gamma=(G,\ell)$ and $\Gamma'=(G',\ell')$. We say that $\Psi$ and $\Psi'$ are \emph{isomorphic} if there is an isomorphism of tropical curves $\Gamma\ra \Gamma'$  whose induced graph isomorphism  $\gamma\col G\ra G'$  satisfies $\gamma_*(P,s)=(P',s')$.

 \subsection{The moduli space of spin tropical curves}
 
We denote by $\Sgnt$ (respectively, $\Sgntb$) the set of isomorphism classes of spin tropical curves (respectively, extended spin tropical curves). We are going to give  them  the structure of extended generalized cone complexes.

Let $\Gamma=(G,\ell)$ be an extended tropical curve. Recall that $SP_G$ denotes the set of spin structures on $G$.
The set of   spin structures on $\Gamma$  is written
$$
S^{\trop}_{\Gamma}:=\{\Psi=(\Gamma, P,s),  \;  \forall (P,s)\in SP_G\}.
$$
Thus $S^{\trop}_{\Gamma}$ is  a finite set,   partitioned according to the parity of a spin structure as follows:
 $
S^{\trop}_{\Gamma}=(S^{\trop }_{\Gamma})^+\sqcup (S^{\trop}_{\Gamma})^- .
$
Notice that $\Aut(\Gamma)$ may act non-trivially  
 on $S^{\trop}_{\Gamma}$.
 
Next, for a stable spin graph $(G,P,s)\in \Spgn$ we consider the subsets
\[
 S^{\trop}_{(G,P,s)}\subset\Sgnt \quad \text{ and } \quad \ov{S}^{\trop}_{(G,P,s)}\subset \Sgntb
\] 
defined as the set 
of isomorphism classes, respectively, of spin tropical curves and of extended spin tropical curves, whose underlying spin graph is isomorphic to $(G,P,s)$. These sets will define a     stratification of $\Sgnt$ and $\Sgntb$.

We construct $\Sgnt$ and $\Sgntb$ as generalized cone complexes using a procedure analogous to, and compatible with,
 the one used to construct $\Mgnt$ and $\Mgntb$.

We consider the category, {\sc{spin}}$_{g,n}$, whose objects are isomorphism classes of stable spin graphs of genus $g$ with $n$ legs and whose morphisms are generated by contractions and by automorphisms of spin graphs. To an isomorphism class of a spin graph $(G,P,s)$ we associate a cone
$$
\sigma_{_{(G,P,s)}}=\RR_{\geq 0}^{E}.
$$
This cone has the natural integral structure determined by the sub-lattice parametrizing tropical curves whose edges have integral length.
We write $\sigma_{_{(G,P,s)}}^o=\RR_{> 0}^{E}$ for its  interior.

To a contraction  $\gamma\col (G,P,s)\to (G',P',s')$ we associate an injection of cones 
$$
\iota_{\gamma}\col \sigma_{_{(G',P',s')}}\ha \sigma_{_{(G,P,s)}}
$$
whose image is the face of $\sigma_{_{(G,P,s)}}$ where the coordinates corresponding to $E(G)\smallsetminus E(G')$ are zero.
If $\gamma$ is an automorphism, then $\iota_{\gamma}$ is the corresponding   automorphism of rational cones.
By our earlier results, this gives a contravariant functor from {\sc{spin}}$_{g,n}$ to the category of rational polyhedral cones.
Therefore we can take the colimit of the diagram of cones $\sigma_{_{(G,P,s)}}$ using the inclusions $\iota_{\gamma}$,
for all morphisms $\gamma$ in {\sc{spin}}$_{g,n}$.
We define this colimit to be the moduli space of $n$-pointed spin  tropical curves of genus $g$:
\[
  \Sgnt = \varinjlim \left(\sigma_{_{(G,P,s)}}, \iota_{\gamma}\right).
\]
 Hence $\Sgnt$ is canonically a generalized cone complex, and the following  is easily seen to hold.
\begin{prop}
\label{Sgtprop}
 The moduli space of spin tropical curves, $\Sgnt$,
 is a topological space of  pure dimension $3g-3+n$.
 We have a stratification
 $$
\Sgnt=\bigsqcup_{[G,P,s]\in [\Spgn]} S^{\trop}_{(G,P,s)}.
$$
Moreover, we have
\[
  S^{\trop}_{(G,P,s)}\cong\sigma_{_{(G,P,s)}}^o/ \Aut(G,P,s),
\]
and $ S^{\trop}_{(G',P',s')}\subset \overline{S ^{\trop}_{(G,P,s)}}$ if and only if $[G,P,s]\geq [G',P',s']$.
\end{prop}

One constructs the moduli space for extended spin tropical  curves in  the analogous way.
To an isomorphism class of a spin graph $(G,P,s)$ we now associate the extended cone
$$
\ov{\sigma}_{(G,P,s)}:=\ov{\mathbb R}^{E}_{\ge0} 
$$
and its interior $\ov{\sigma}_{_{(G,P,s)}}^o=({\RR}_{>0}\cup \{\infty\})^{E}$.
The rest  of the construction is the same, and yields
  the moduli space of extended spin tropical curves
\[
  \Sgntb = \varinjlim \left(\ov{\sigma}_{(G,P,s)}, \iota_{\gamma}\right),
\]
whose stratification we denote
\[
\Sgntb=\bigsqcup _{[G,P,s]\in [\Spgn]} \ov{S}^{\trop}_{(G,P,s)}\cong\bigsqcup _{[G,P,s]\in [\Spgn]}\ov{\sigma}_{_{(G,P,s)}}^o/ \Aut(G,P,s).
\]

\begin{prop}
\label{Sgtbprop}
 The moduli space of extended spin tropical curves,
 $\Sgntb$, is a generalized extended cone complex, and a  topological space of pure dimension $3g-3+n$ containing $\Sgnt$ as a  dense open subset.
It has two connected components:
$$
\Sgntb=\Sgntbo \sqcup \Sgntbe
$$
corresponding to odd and even spin tropical curves.
\end{prop}
 
\begin{proof}
With respect to what we already said, 
the only thing that needs to be proved is  that $\Sgntbo$ and $ \Sgntbe$ are connected.
This follows by  the same argument we used to prove Proposition \ref{pospin} \eqref{pospinb}.
\end{proof}


 We have a natural map 
\[
\pi^{\trop}\col   \Sgntb\la  \Mgtnb
\]
sending the point parametrizing a  spin tropical curve $(\Gamma,P,s)$, with 
$\Gamma=(G,\ell)$, 
  to the point parametrizing the   tropical curve $(G,\pi^{\trop}(\ell))$, where 
\[
\pi^{\trop} (\ell)(e):=
\begin{cases}
\begin{array}{ll}
2\ell(e), & \text{if } e\in E\smallsetminus P;\\
\ell(e),  & \text{otherwise}.
\end{array}
\end{cases}
\]
We consider also the restrictions (with self-explanatory notation)
\[
    \pi^{\trop,-}\col\Sgntbo\la \Mgtnb \quad \text{ and } \quad \pi^{\trop,+}\col\Sgntbe\la\Mgtnb.
\]

Next we prove that $\pi^{\trop}$, and hence $\pi^{\trop,-}$ and $\pi^{\trop,+}$, are  morphisms of extended generalized cone complexes in the sense of \cite[Sect.  2]{ACP15}.

\begin{proposition}\label{prop:pitrop}
The map $\pi^{\trop}$ is a morphism of extended generalized cone complexes and for every extended tropical curve $\Gamma$ we have
\[
(\pi^{\trop})^{-1}([\Gamma])\cong S^{\trop}_\Gamma/\Aut(\Gamma).
\]
\end{proposition}

\begin{proof}
Obviously $\pi^{\trop}$ is   compatible with the cone diagrams defining  $\Mgtnb$ and  $\Sgntb$, and
  the induced map from $\sigma_{(G,P,s)}$ to $\sigma_G$ is a morphism of rational polyhedral cones with integral structure
(it is the restriction of the integral linear transformation $T\col\mathbb R^{E}\ra\mathbb R^{E}$ defined on the basis $\{e : e\in E\}$  as $T(e)=2e$ if $e\in E\smallsetminus P$, and $T(e)=e$ if $e\in P$). Hence $\pi^{\trop}$ is a map of extended generalized cone complexes.

Let $\Gamma=(G,\ell)$ be an extended tropical curve. 
There is a natural map $\rho\col S^{\trop}_\Gamma\ra \ol{S}_{g,n}^{\trop}$ taking $(\Gamma,P,s)$, with $(P,s)\in SP_G$, to the class of $(\Gamma_P,P,s)$, where $\Gamma_P=(G,\ell_P)$, $\ell_P(e)=\ell(e)/2$ for $e\in E\setminus P$, and $\ell_P(e)=\ell(e)$ for $e\in P$. We have $(\pi^{\trop})^{-1}([\Gamma])=\text{Im}(\rho)$. Given $(P,s)$ and $(P',s')$ in $SP_G$, we have $\rho(\Gamma,P,s)=\rho(\Gamma,P',s')$ if and only if there is an automorphism of $\Gamma$ inducing an isomorphism between $(\Gamma,P,s)$ and $(\Gamma,P',s')$. 
 Therefore $\text{Im}(\rho)\cong S^{\trop}_\Gamma/\Aut (\Gamma)$.
\end{proof}

\section{Algebraic stable spin curves}\label{sec:algspin}
 
  \subsection{Stable spin curves and their moduli space}
  \label{subspin}
We let  $\Sgnbst$ be the moduli stack of stable $n$-pointed spin curves of genus $g$ and  $\ma S_{g,n}\subset \Sgnbst$ be the substack parametrizing theta characteristics on smooth curves.  
Given a $k$-scheme $B$, a section in $\Sgnbst(B)$ is a pair  $(\wh{\ma X}, \wh{\ma L})$, described as follows. We have a  flat morphism $f\col \ma{\wh X}\ra B$  
  such that for every  $b\in B$ the fiber $\wh{ X}_b$ is a genus-$g$
 quasistable $n$-pointed  curve. Next,  $\wh{\ma L}$ is a line bundle on $\wh{\ma X}$  
endowed with a homomorphism $\alpha\col \wh{\ma L}^{\otimes 2}\ra \omega_f$ such that for every     $b\in B$ the map $\alpha _{|\wh {  X}_b}\col \wh{\ma L} _{|\wh{X}_b}^{\otimes 2}\ra \omega_{\widehat{ X_b}}$ is an isomorphism away from the  exceptional components,  and  $\wh{\ma L}_{|E}\cong \mathcal O_E(1)$ for every exceptional component $E$ of $\wh{  X}_b$.

We have a natural, representable morphism 
$$
\pi\col \Sgnbst\la \Mgnbst
$$
sending $(\wh{\ma X}, \wh{\ma L})$ to the stable model   of $\wh{\ma X}$.
The morphism $\pi$ is finite of degree $2^{2g}$,
hence the fiber of $\pi$ over a point parametrizing a   curve $X$, written  ${\cS}_X$,  has dimension $0$ and  length $2^{2g}$. 
With the notation \eqref{SZ0}, we have ${\cS}^0_X\subset {\cS}_X$ with equality if and ony if $X$ is nonsingular (see below).

Recall the notation in \eqref{eq:XtildeR}. A point in ${\cS}_X$ parametrizes the isomorphism class of a pair
  $(\hX_R, \hL_R)$, where $\hX_R$ is the quasistable      curve associated to a set of nodes $R$ of $X$ and $\hL_R$ is a line bundle on $\hX_R$ such that
\begin{enumerate}[(1)]
 \item 
the restriction,  $L _R$, of  $\wh L_R$ to $X^{\nu}_R$ satisfies $L_R^2\cong \omega_{X^{\nu}_R}$;
 \item
 the restriction of $\hL_R$  to $E_r$  is $\cO_{E_r}(1)$ for every $r\in R$.
   \end{enumerate}
    Notice that
   $(\hX_R, \hL_R)$ depends only on $L_R$ (i.e.  different gluings over the nodes lying on  exceptional components of $\hX_R$ give the same point in ${\cS}_X$).
 The first requirement   implies that the dual graph of ${X^{\nu}_R}$ is cyclic. Therefore we can alternatively
 describe    
the points   of ${\cS}_X$ as parametrizing
   pairs  $(R, L_R)$ defined as follows: 
 
\begin{enumerate}[(1)]
 \item
 $R\subset E$   such that $E\smallsetminus R$ is cyclic;
  \item
 $L_R\in \Pic(X^{\nu}_R)$  such that $L_R^2\cong \omega_{X^{\nu}_R}$.
 \end{enumerate}

We earlier used the notation  $(\hX_R, \hL_R)$  for a point in $\Sgnbst$, but
 we can, and will,      denote the same point    by a pair
$(X^{\nu}_R, L_R)$.

Now, $\Sgnbst$ has two   connected and irreducible  components, denoted
$$
\Sgnbsto\la \Mgnbst  \quad \quad  \text{and} \quad \quad \Sgnbste\la \Mgnbst
$$
parametrizing, respectively, odd and even spin curves, where the parity refers to the parity of $h^0(\hX_R, \hL_R)$
(equivalently, the parity of $h^0( X^{\nu}_R,  L_R)$)
for every $(\hX_R, \hL_R)$ parametrized by $\Sgnbst$. 

We   write, with self-explanatory notation,
${\cS}_X={\cS}_X^-\sqcup {\cS}_X^+$.
One refers to 
${\cS}_X$ as the space of (stable) spin structures   on $X$, with  ${\cS}_X^-$ and ${\cS}_X^+$ parametrizing odd and even ones.
\subsection{The dual graph of a stable spin curve}

Let $X$ be a stable   curve, $G$ its dual graph, and $R\subset E$ a set of nodes of $X$; write $P  =E\smallsetminus R$.
Then the connected components of $X^{\nu}_R$ are in natural bijection with the vertices of the graph
$G/P$. We thus write
$
 X^{\nu}_R=\sqcup_{v\in V(G/P)} Z_v
$
with $Z_v$ connected.  Recall, from \ref{sec:algcur}, that we view $ X^{\nu}_R$ as a pointed curve so that
 the dual graph of $Z_v$ is the graph  $\ov{P}_v$   defined in Section \ref{sec:spingraphs}, and 
 the weight of $v$ in $G/P$ is equal to the genus of $Z_v$.
\begin{defi} 
\label{dsg}
 The {\it dual spin graph} of a spin curve $(X^{\nu}_R, L_R)$
 is the spin graph $(G,P ,s )$ defined as follows:
\begin{itemize} 
\item  $G$ is the dual graph of $X$;
\item $P =E \smallsetminus R$;
\item $s(v)$ is the parity of $h^0(Z_v, (L_R)_{|Z_v})$,   for every $v\in V(G/P)$. 
\end{itemize}
\end{defi}

We need to check that we defined an actual  spin graph.
As we said above, if $(X^{\nu}_R, L_R)$ is a spin curve, then $E \smallsetminus R \in \cC_G$.
Next, if $v\in V(G/P)$ has weight zero, then $Z_v$ is smooth and has genus zero, and     
$(L_R)_{|Z_v}$ has degree $-1$, and so 
 $s(v)=h^0(Z_v, (L_R)_{|Z_v})=0$.

\begin{remark}\label{rem:chm}
A spin curve and its dual spin graph have the same parity.
Indeed, the parity of  $(X^{\nu}_R, L_R)$ is the parity of $h^0(X^{\nu}_R, L_R)$, and we have
$$
h^0(X^{\nu}_R, L_R)=\sum_{v\in V(G/P )} h^0(Z_v, (L_R)|_{Z_v}) \equiv \sum_{v\in V(G/P)}s (v) \  \text{ mod} (2).
$$ 
\end{remark}

\begin{prop}
\label{Spingraphprop}
We have a commutative diagram of surjective maps
  $$
\xymatrix@=.4pc{
&&&&\Sgnbst  \ar[dd]_{\pi}    \ar[rrrr]^{\psi}  &&&& [\Spgn]\ar[dd] &&&   &&&&  \\
\\
 &&&&\Mgnbst  \ar[rrrr] &&&&\Sgn
}
$$
where the two horizontal arrows  map 
 a stable (respectively  spin) curve to the class of its dual (respectively dual  spin) graph, and the right vertical arrow 
 maps $[G,P,s]$ to $G$. 
The same holds by replacing $\Sgnbst$ and $[\Spgn]$ with 
    $\Sgnbst^+$ and $[\Spgn^+]$ or,   if $g>0$, with $\Sgnbst^-$  and  $[\Spgn^-]$ .
\end{prop}
 
\begin{proof}
 The only non-evident claim is that the top horizontal arrow is surjective.
 Let $(G,P,s)$ be a spin graph. Pick any pointed curve $X$ dual to $G$.
 Set $R=E\smallsetminus P$ and let $X^{\nu}_R$ be the normalization of $X$ at $R$.
As above, we write $\{Z_v,\;  v\in V(G/P)\}$ for the connected components of $X^{\nu}_R$.
To conclude, we need to show that for every $v\in V(G/P)$ we can pick $L_v\in  \Pic(Z_v)$ such that  $L_v^2\cong \omega_{Z_v}$
and whose parity is that of $s(v)$.
For that, consider 
$ 
{\cS}_{Z_v}^0=\{L \in \Pic(Z_v):\  L ^2\cong \omega_{Z_v}\}$.

If $Z_v$ has positive genus, then ${\cS}_{Z_v}^0$ contains  elements of even and odd parity. Hence we can choose 
an element whose parity is the same as $s(v)$.
If $Z_v$ has genus zero, then ${\cS}_{Z_v}^0$ consists of one element  which, having degree $-1$, is necessarily even. But in this case $s(v)=0$, so we are done.
\end{proof}

Let $X$ be a  stable curve and let us  study the structure of   ${\cS}_X$.
Consider the diagram in Proposition~\ref{Spingraphprop}, and let us restrict the map $\psi$ to 
${\cS}_X$.
If  $(G,P,s)\in \Spgn$ we denote by  
$${\cS}_{(X,P,s)}:=\psi^{-1}([G,P,s])\cap {\cS}_X, 
$$ 
the subscheme of spin structures on $X$ whose dual graph is isomorphic to $(G,P,s)$.
We thus have a decomposition
$
 {\cS}_X=\sqcup_{(P,s)\in  [SP_G]}{\cS}_{(X,P,s)}.
$

As we shall state below, the structure of  ${\cS}_X$ at loci of the form  ${\cS}_{(X,P,s)} $
 does not depend on the parity function. It is thus convenient to consider, for any  $P\in \cC_G$,
 the following subscheme of ${\cS}_X$ 
 $$
 {\cS}_{(X,P)}=\sqcup_{s}{\cS}_{(X,P,s)}
 $$
where $s$ varies over all parity functions.

 \begin{remark}
 \label{SXP}
If $G$ is cyclic we have  $ 
 {\cS}_{(X,G)}=\cS_X^0
 $  (recall     \eqref{SZ0}). 
For an arbitrary $G$  there is a  bijection between the points
 of ${\cS}_{(X,P)}$ and the set of theta characteristics, ${\cS}^0_{X^{\nu}_R}$, on $X^{\nu}_R$.
\end{remark}

 We have  ${\cS}_{(X,P)}={\cS}^+_{(X,P)} \sqcup {\cS}^-_{(X,P)}$, self-explanatorily.
We now combine some results of \cite[Sect. 1]{CC03}  with \cite[Cor. 2.13]{Har82}.

\begin{fact} 
\label{CCK60}
Let $X$ be a stable   curve of genus $g$ and $G$  its dual graph. Set $b=b_1(G)$ and $|w|=\sum_{v\in V}w(v)$.
Let  $P\in \cC_G$. Then
\begin{enumerate}[(a)]
\item 
$ |{\cS}_{(X,P)}|=2^{b_1(P)+2|w|}
$   and the length
of ${\cS}_X$ at every point of ${\cS}_{(X,P)}$ is equal to $2^{b-b_1(P)}$, so that
$ 
\operatorname{length}{\cS}_{(X,P)}=2^{b+2|w|}.
$ 
  \item If $P$ is connected and $b_1(P)\ne 0$, then  
\[
|{\cS}^-_{(X,P)} |=|{\cS}^+_{(X,P)}|=2^{b_1(P)+2|w|-1}.
\]
  \end{enumerate}
 \end{fact}

Now, using Remark~\ref{SXP} we have
the following stratification, highlighting  the recursive structure of ${\cS}_X$  
by expressing its  boundary, ${\cS}_X\ssm  {\cS}^0_X$, in terms of   the theta-characteristics on the partial normalizations of $X$
\[
{\cS}_X^{\red}=\bigsqcup_{P\in \cC_G}{\cS}_{(X,P)}^{\red}
\cong\bigsqcup_{E\ssm R\in \cC_G}{\cS}_{X^{\nu}_R}^0.
\]

\subsection{Families of stable spin curves}
A {\it one-parameter family of curves} is a   family of (pointed) nodal curves $f\col\cX\to B$ 
over a   regular, connected curve $B$ with a marked point, $b_0\in B$; we denote by $X$ the fiber over $b_0$
and by $G$ its dual graph.
We  will always  assume that every fiber  over $B\smallsetminus \{b_0\}$  has the same dual graph, denoted by $H$;
we shall denote by $Y$ any such fiber,  which we call the ``generic'' fiber.
To   such  a family $f$ we associate its {\it dual contraction}:
$$
\gamma_f\col G\la H=G/S_0
$$
where $S_0\subset E(G)$ are the nodes of $X$ which are not specializations of nodes of $Y$.
We shall write $\gamma=\gamma_f$ when no confusion can occur.

Assume the above family is polarized, 
i.e.,  we have a line bundle, $M$, on $Y$ specializing to a line bundle,
$L$, on $X$.  Then by \cite[Prop. 4.3.2]{CC18} we have an identity of divisors on $H$ 
\begin{equation}
 \label{mdeg}
\mdeg M=\gamma_*\mdeg L.
\end{equation}
We shall write $(\cX, \cL)\to B$ for such a polarized family, with $\cL$ a family of line bundles on the fibers of $f$
restricting to $L$ on  $X$  and to $M$ on  $Y$.

A {\it one-parameter family of stable spin curves} is a polarized family, written $(\hcX,\wh{\ma L})\to B$, such that
  $\hf\col\hcX\to B$ is a one-parameter family of quasistable   curves as defined above (we write $\hX$, respectively $\hY$,  for its special, resp. generic,  fiber)    and such that for every $b\in B$ the fiber,  $(\hX_b,\wh{\ma L}_{|\hX_b})$, is a stable spin curve. 
Let    $f\col\cX\to B$ be the  stable model of $\hf$.
Then  $\hX$ is the blow-up of   $X$   at a set of nodes  $R$ of $X$, and its  dual graph is written $\hG_R$. 
 
Similarly, let $T$ be the set of nodes of $Y$  that are blown-up in $\hY$ and denote by $\wh{H}_T$ the 
dual graph of $\hY$.
 Since every exceptional component of $\hY$ specializes to an exceptional component of $\hX$, we have  (recall that $T\subset E(H)=E(G)\smallsetminus S_0$)
\begin{equation}
 \label{RS}
 T\subset R\smallsetminus S_0.
\end{equation}

Let  $\widehat{\gamma} \col \hG_R\to \hH_T$  be the dual contraction of $\hf$.
We have a
 commutative diagram of contractions, where the vertical arrows correspond to the stabilizations   $\hX\to X$ and $\hY\to Y$
\begin{equation}
\label{gammadiag}
\xymatrix{
\widehat{G}_R \ar_{st}[d]\ar[r]^{\widehat\gamma}& \widehat{H}_T\ar^{st}[d]\\
G\ar[r]^{\gamma}& H 
}
\end{equation}

\begin{lemma}
\label{purity}
Let $(\hcX,\wh{\ma L})\to B$  be a one-parameter family of stable spin  curves and   $f\col \cX\to B$ its stable model. 
Then, with the above notation,  $T=R\smallsetminus S_0$.
\end{lemma}
\begin{proof}
By \eqref{RS} we need to prove that every $r\in R\smallsetminus S_0$ lies in $T$. 
This is equivalent to saying that every exceptional component, $E$,  of the special fiber, $\hX$,
is either the specialization  of an exceptional component of the generic fiber,
or neither of the two nodes lying in $E$ is specialization of a node of the generic fiber.

This is a special case of    the explicit description of the deformation space given in \cite[Subsect. 3.2]{CCC07}.
At the beginning of that subsection it is shown that, if $E$ is an exceptional component of the special fiber,
then either both nodes of $E$ are preserved, or no node  of $E$ is preserved.
This is precisely what is needed here.
   \end{proof}
 \subsection{The stratification of $\Sgnbst$}\label{secfam} 
The goal of this subsection is to describe the stratification of   $\Sgnbst$.
The strata of $\Mgnbst$ are the loci, $\cM_G$, parametrizing  curves whose dual graph is isomorphic to $G$, for $G$ varying in $\Sgn$. 
In \cite[Subsect. 3.4]{ACP15} the strata $\cM_G$ are given an explicit presentation
 
\begin{equation}
 \label{Mgnt}
\wt{\cM}_G:=\prod_{v\in V(G)}\cM_{w (v),\deg (v)+\ell (v)}\la \cM_G=\left[\wt{\cM}_G/\text{Aut}(G)\right].
\end{equation}
In particular,   $\cM_G$ is irreducible. It is well known that $\Mgnbst=\sqcup_{G\in \Sgn}\cM_G$  is a graded stratification.
We shall write $\cM_v:=\cM_{w (v),\deg (v)+\ell (v)}$.

Consider now the morphism $\pi\col \Sgnbst\ra \Mgnbst$, and set
$$
 {\cS}_G:=\pi^{-1}(\cM_G).
$$
Now, given $P\in \cC_G$ and $(P,s)\in SP_G$ we let
${\cS}_{(G,P)}$ be the substack of $\Sgnbst$ parametrizing stable spin curves whose dual spin graph 
has  support  isomorphic to $(G,P)$, and 
  ${\cS}_{(G,P,s)}$  be the substack   where the dual spin graph is isomorphic to $(G,P,s)$.
  Then we have
  \begin{equation}
 \label{oSGStrata}
 {\cS}_G=\sqcup_{P\in [\cC_G]}{\cS}_{(G,P)}=\sqcup_{(P,s)\in [SP_G]}{\cS}_{(G,P,s)}.
 \end{equation}
These decompositions are the extensions over the moduli spaces of the ones described   earlier for a fixed curve.
Notice that if we have two spin structures, $(P,s)$ and $(P',s')$, on $G$ and an automorphism of $G$ mapping one to the other, then, by definition,
$ 
{\cS}_{(G,P,s)}= {\cS}_{(G,P',s')}.
$ 
We have thus  a decomposition
\begin{equation}
 \label{Sgnstrat}
\Sgnbst=\bigsqcup_{[G,P,s]\in [\mathcal{SP}_{g,n}]}\ma S_{(G,P,s)}.
\end{equation}

  \begin{example}
\label{basicex}
Recall that we view $G$ as an element in $\cE_G$.
The stratum ${\cS}_{(G,G)}$ is not empty if and only if $G$ is cyclic. Let us assume this is the case.
Then ${\cS}_{(G,G)}$ parametrizes theta-characteristics on the curves in $\cM_G$, i.e.
$$
{\cS}_{(G,G)}=\{(X,L) :    X\in \cM_G, \   L\in \cS^0_X\}.
$$
\end{example}

\begin{prop}\label{thm-strata}
Let $(G,P,s)$ and $(H,Q,s')$ be  stable spin graphs. Then 
the following are equivalent.
\begin{enumerate}[(a)]
 \item
 \label{thm-strata1}
$\cS_{(G,P,s)}\cap \overline{\cS}_{(H,Q,s')}\neq \emptyset$.
\item
 \label{thm-strata2}
 $[G,P,s]\geq [H,Q,s']$.
\item
 \label{thm-strata3}
$\cS_{(G,P,s)}\subset \overline{\cS}_{(H,Q,s')}$.
\end{enumerate}
\end{prop}
\begin{proof}
We denote, as before,  $R=E(G)\smallsetminus P$ and $T=E(H)\smallsetminus Q$.

 \eqref{thm-strata1} $\Rightarrow$  \eqref{thm-strata2}.
By hypothesis
there exists  a  one-parameter family of stable spin curves $(\hcX,\wh{\ma L})\to B$   
whose generic fiber, $(\hY,  \cL _{\hY}) $, 
has dual spin graph  $(H,Q,s')$   and whose special fiber, $(\hX,  \cL _{\hX})$, has dual spin graph  $(G,P,s)$. 
Recall that, as described in  the commutative diagram \eqref{gammadiag}, we have 
a contraction $\gamma\col G\ra H$. 
We need to show that $\gamma_*(P,s)=(Q,s')$. 
From Lemma~ \ref{purity} we derive that $\gamma_*P=Q$.
The fact that  $s'(v')=\sum_{v\in\ol{\gamma}^{-1}(v')}s(v)$ is a consequence of the deformation invariance of the parity of spin curves.

  \eqref{thm-strata2} $\Rightarrow$  \eqref{thm-strata3}.  Consider a contraction $\gamma\col G\to H$ such that  $\gamma_*(P,s)=(Q,s')$. Hence  $\gamma$ induces a contraction
 $\widehat \gamma\col\widehat{G}_{R}\to \widehat{H}_T$
 and we can form the commutative diagram \eqref{gammadiag}. 
 Now, let $X$ be a pointed curve dual to $G$. The existence of the contraction $\gamma$ implies the existence of a one-parameter family
 of stable   curves
 $f\col\ma X\to B$ whose generic fiber, $Y$, is dual to $H$ and whose special fiber is $X$.
  Therefore the space, ${\cS}_Y$,  of spin structures on $Y$, 
specializes to ${\cS}_X$. 
We know that every  $[(\hX ,\wh L )]\in {\cS}_{(X,P,s)}$ is the specialization of some spin structure on $Y$
 lying in some  stratum ${\cS}_{(Y,\ol Q, \ol s')}$; using Lemma \ref{purity} as before, we get $\gamma_*(P,s)=(\ol Q,\ol s')$. Since $\gamma_*(P,s)=(Q,s')$ we obtain  $(\ol Q, \ol s')=(Q,s')$,
 as wanted.
As  \eqref{thm-strata3} $\Rightarrow$  \eqref{thm-strata1} is obvious, we are done.
   \end{proof}

We will prove in Theorem~\ref{cor:irr} that $\ma{S}_{(G,P,s)}$ is irreducible, and hence, combining with the previous proposition, we will obtain the following.

\begin{thm}\label{thm:strata}
Decomposition \eqref{Sgnstrat}  is a graded stratification of $\Sgnbst$.
\end{thm}


\subsection{Local analysis}
 \label{rem:defospin}
We here recall the description of the versal deformation space of a   pointed spin curve and of the morphism $\pi\col \Sgnbst\ra \Mgnbst$. More details can be found in \cite{ACG11} and \cite{CCC07}.

Let $x\in \ol{\cM}_{g,n}$ and $y\in \Sgnbst$ be closed points such that  $x=\pi(y)$. Assume that $x$ represents the  curve $X$ with dual graph $G$ and that $y$ represents the   spin curve $(\wh X, \wh L)$. Let $E=\{e_1,\dots,e_\delta\}$ be the nodes of $X$, and assume that the exceptional components of $\wh X$ lie over the nodes $e_1,\dots,e_r$, for some $r\in\{0,\dots,\delta\}$.  
There are  \'etale local coordinates $t_1,\dots,t_{3g-3+n}$ such that 
\begin{equation}\label{eq:m}
\wh {\ma O}_{\ol{\cM}_{g,n},x}\cong k[[t_1,\dots,t_{3g-3+n}]],
\end{equation} 
where the vanishing of $t_i$ is the divisor corresponding to the deformations of $X$ that are trivial locally at the node $e_i$, for $i=1,\dots,\delta$.

Similarly, there are  \'etale local coordinates $s_1,\dots,s_{3g-3+n}$ such that 
\[
\wh{\ma O}_{\ol{S}_{g,n},y}\cong k[[s_1,\dots,s_{3g-3+n}]].
\]
By Lemma~\ref{purity} we can assume that  for $i=1,\dots,r$, respectively, for $i=r+1,\dots,\delta$, the vanishing of $s_i$ is the locus  of the deformations of $(\wh X, \wh L)$ that are trivial locally at the nodes contained in the exceptional component of $\wh X$ lying over $e_i$, respectively, are trivial locally at the node $e_i$  seen as a node of $\wh X$. 

Locally at $x$ and $y$, the   morphism $\pi\col \Sgnbst\ra \Mgnbst$ is induced by the ring homomorphism $\pi^\#\col\wh {\ma O}_{\ol{\cM}_{g,n},x}\ra \wh {\ma O}_{\Sgnbst,y}$ given by $\pi^\#(t_i)=s_i^2$ for $i\le r$, and $\pi^\#(t_i)=s_i$ for $i>r$.

Let now $G$ be a stable graph.
The closures in $\Mgnbst$ and in 
$\Sgnbst$
of  the   loci
$\cM_G$, ${\cS}_G$ and  ${\cS}_{(G,P,s)}$ will be written, respectively,
$\ov{\cM}_G$,  $\ov{\cS}_G$ and $\ov{\cS}_{(G,P,s)}$.  We have $\ov{\cS}_G=\pi^{-1}(\ov{\cM}_G)$.

Let $G$ and $G'$ be two stable graphs such that ${\cM}_{G'}\subset  \ol{\cM}_G$. We need to describe the structure of  $\ol{\cM}_G$
at a point $x'\in {\cM}_{G'}$.

 Let $X'$ be a stable curve parametrized by $x'$. 
Write  $E(G')=\{e_i\}_{1\le i\le \delta'}$ so that $\delta' \geq \delta$. Let $J$ be the set of all contractions of $G'$ to $G$.

As explained in the proof of \cite[Prop.  XII.10.11]{ACG11}, the  \'etale  local picture of   $\ol{\cM}_G$   at the point $x'$ is the one of $|J|$ linear subspaces in $k^{3g-3+n}$ of dimension $3g-3+n-\delta$ meeting transversally at the origin. 
Now, as in  \eqref{eq:m}, we have  
$
\wh{\ma O}_{\ol{\cM}_{g,n},x'}\cong k[[t_1,\dots,t_{3g-3+n}]]. 
$  
Given a contraction  $\gamma \col G'\to G$ in $J$,  we consider the ideal $\ma I_\gamma=(t_{\gamma_1},\dots,t_{\gamma_\delta})$, where $\{e_{\gamma_1},\dots,e_{\gamma_\delta}\}$ is the image of $E(G)$ in the   injection $E(G)\subset E(G')$ induced by $\gamma$. Then we have
$$
\wh{\ma O}_{\ol{\cM}_G, x'}\cong\frac{k[[t_1,\dots,t_{3g-3+n}]]}{\prod_{\gamma\in J} \ma I_\gamma}=\frac{k[[t_1,\dots,t_{\delta'}]]}{\prod_{\gamma\in J} \ma I_\gamma}\otimes k[[t_{\delta'+1},\dots,t_{3g-3+n}]].
$$
Let
$$
\nu:\ol{\cM}^\nu_{G}\la \ol{\cM}_G
$$
be
the normalization  of $\ol{\cM}_G$. Then $\ol{\cM}^\nu_{G}$ is smooth and, locally over $x'$, is the disjoint union of the $|J|$ branches described by the ideals $\ma I_\gamma$. Hence, given $\gamma\in J$, if we let $x'_\gamma$ be the point in $\ol{\cM}^\nu_G$ lying on the branch corresponding to $\ma I_\gamma$, we have 
\begin{equation}\label{eq:Igamma}
 \wh{\ma O}_{\ol{\cM}^\nu_G, x'_\gamma}\cong\frac{k[[t_1,\dots,t_{3g-3+n}]]}{\ma I_\gamma}.
\end{equation}

We define
\[
\ol{\wt{\cM}}_G:=\prod_{v\in V(G)} \ol{\cM}_{w(v),\deg(v)+\ell(v)}.
\]
The normalization $\ol{\cM}^{\nu}_G$ of $\ol{\cM}_G$ is given an explicit presentation 
\begin{equation}\label{eq:tildebarra}
\chi\col\ol{\wt{\cM}}_G\la\ol{\cM}^{\nu}_G
\end{equation}
such that  $\ol{\cM}^{\nu}_G=[\ol{\wt{\cM}}_G/\Aut(G)]$ (see \cite[Prop. XII.10.11]{ACG11}).
 \subsection{Presentations}\label{sec:presentation}
 Let $(G,P,s)$ be a stable spin graph.
Let $\ma  X_G\to \cM_G$ be the universal curve   and $\wt{\ma X}_G$ its base change by the map
$ \wt{\cM}_G\to \cM_G$.
Set  
 $\wt{\ma S}_{(G,P,s)}:=\ma S_{(G,P,s)}\times_{\cM_G}\wt{\cM}_G.$ 

We consider first the case  $P=G$, and look at the
  Cartesian diagram
\begin{equation}\label{diag:univ}
\SelectTips{cm}{11}
\begin{xy} <16pt,0pt>:
\xymatrix{ 
\wt{\ma X}_G \ar[r]\ar[d]& \wt{\cM}_G\ar[d]&\tcS_{(G,G,s)} \ar[l]\ar[d]\\
 \ma X_G \ar[r]&  \ma{M}_G=[\wt{\cM}_G/\Aut(G)]&\ma S_{(G,G,s)}\ar[l]  \\
}
\end{xy}
\end{equation}
  
We can view $\tcS_{(G,G,s)}$ as the moduli space of theta characteristics of the fibers of $\wt{\ma X}_G \ra\wt{\cM}_G$, whose  parity function is prescribed by $s(v)$, where $v$ is the unique vertex of $V(G/G)$. 

In case $P\neq G$   we need to  introduce a new space to have a suitable presentation of $S_{(G,P,s)}$. Set $R=E\ssm P$.   Recall that the set of    connected components of $(G,P,s)$ is written   $\{(\ov P_v,s_v), \  v\in V(G/P)\}$.  
Since  $\ov P_v$
is a  stable graph  (with legs), we can consider, for every $v$, the substacks 
$$ 
\cM_{\ov P_v}\subset \ol{\cM}_{g_v, \ell_v+d_v}
\;\text{ and }\;
\ma S_{(\ov P_v,\ov P_v,s_v)} \subset \ov{\cS}_{g_v, \ell_v+d_v}
 $$ 
where $g_v$,  $\ell_v$ and $d_v$ are   as follows. 
Set 
$g_v=g(\ov P_v)$ and 
$ 
\ell_v =\sum_{u\in V(\ov P_v)}\ell_G(u)$, and $ d_v=\sum_{u\in V(\ov P_v)}\deg_R(u)
$ 
where $\ell_G(u)$ is the number of legs of $G$ ending at $u$, and $\deg_R(u)$ is the degree of $u$ as a vertex of the subgraph $\langle R\rangle$ of $G$. As $G\in \Sgn$, we have
$ 
\sum_{v\in V(G/P)}\ell_v=n$ and $\sum_{v\in V(G/P)}d_v=2|R|$ and $\sum_{v\in V(G/P)}(g_v-1)+|R|+1=g$.
 We define   
\begin{equation}\label{eq:Sprimedef}
\hcS_{(G,P,s)} := \prod_{v\in V(G/P)} \tcS_{(\ol P_v,\ol P_v,s_v)}.
\end{equation}
Of course, we have $\hcS_{(G,G,s)} =\tcS_{(G,G,s)}$.
 
 Consider the group $\Aut(\oP)= \prod_{v\in V(G/P)}\Aut(\oP_v)$ and recall that
its elements  fix every leg of $\oP$.
Therefore each of them  induces an automorphism of $G$ which fixes every edge in $R$ by  fixing its half-edges,
and maps every component, $\oP_v$, to itself.
Hence we have an  injection
$
\Aut(\oP)\ha \Aut(G,P,s).
$

\begin{remark}\label{modPbar}
The natural map
$$
\hcS_{(G,P,s)}  = \prod_{v\in V(G/P)} \tcS_{(\ol P_v,\ol P_v,s_v)}\la \prod_{v\in V(G/P)} \ma S_{(\ol P_v,\ol P_v,s_v)} 
$$
 is the quotient of the natural action of $\Aut(\oP)$ on $\hcS_{(G,P,s)} $.
 Indeed, we know already that $\ma S_{(\ol P_v,\ol P_v,s_v)} =[ \tcS_{(\ol P_v,\ol P_v,s_v)}/\Aut(\oP_v)]$.
As we said above,  the elements of  $\Aut(\oP)$       leave  every factor of $\hcS_{(G,P,s)} $ invariant.
 \end{remark}

 We denote by    $\Aut(G/P,s)$ the group of automorphisms of $G/P$ 
 which preserve the sign function, $s$, on $V(G/P)$; in symbols
$$
\Aut(G/P,s)
 =\{\alpha\in \Aut(G/P): \ \alpha(v)=u \Rightarrow s(v)=s(u),\  \forall v,u\in V(G/P)\}.
$$ 
Now,  there is a natural homomorphism  
\begin{equation}
 \label{homo}
\Aut(G,P,s) \la \Aut(G/P,s).\end{equation}
Indeed,   assume, for simplicity   $G$   free from legs, then
any $\alpha\in \Aut(G,P,s)$ maps $R$ to itself,  hence $\alpha$ maps the set of half-edges of $R$ to itself.  
 Now the set of half-edges of $R$ is the set of half-edges of $G/P$, hence 
 we have a bijection  $\ov{\alpha}:H(G/P)\to H(G/P)$ mapping $h$ to $\alpha(h)$. It is trivial to check that $\ov{\alpha}$ gives an  automorphism  of $G/P$ as wanted.

We denote by $\Aut_G(G/P,s)\subset \Aut(G/P,s)$ the image of the map \eqref{homo}.
\begin{lemma}
\label{sequence}
We have an exact sequence of groups
 $$
 0\la \Aut(\oP)\la \Aut(G,P,s) \la \Aut_G(G/P,s) \la 0
 $$
\end{lemma}
 
\begin{proof}  Let $\alpha\in \Aut(G,P,s)$.
 If $\ov{\alpha}$ is the identity of $G/P$, then $\alpha$ fixes every half-edge  of $R$ and maps every $\oP_v$ to itself.
 Therefore $\alpha \in \Aut(\oP)$. Next,  every  $\alpha\in \Aut(\oP)$  clearly induces the identity of $G/P$.
\end{proof}

\begin{prop}\label{lem:commute}
For every stable spin graph $(G,P,s)$  we have
\[
 \ma S_{(G,P,s)}=\left[\frac{\hcS_{(G,P,s)}}{\Aut(G,P,s)}\right].
\] 
\end{prop}

\begin{proof}
By Remark~\ref{modPbar} we have 
$$
 \hcS_{(G,P,s)}\la \left[\frac{\hcS_{(G,P,s)}}{\Aut(\oP)}\right]=\prod_{v\in V(G/P)} \ma S_{(\ol P_v,\ol P_v,s_v)} .
$$
We begin by defining   a   map
\[
\theta\col \prod_{v\in V(G/P)} \ma S_{(\ol P_v,\ol P_v,s_v)}\la \ma S_{(G,P,s)}.
\]

Let $B$ be a $k$-scheme.  For any $v\in V(G/P)$, let
  $(\ma Z_v,\sigma_v, \ma L_v)$ 
  be a section in $\ma S_{(\ov P_v,\ov P_v,s_v)}(B)$, so that the dual graph of the fibers of $(\ma Z_v,\sigma_v)\to B$ is $\ov P_v$. 

 Set $R=E \smallsetminus P$. Let $\sigma_v^R\subset \sigma_v$ be the set of sections  that do not correspond to the legs of $G$,
and set $\sigma^R=\cup_{v\in V(G/P)}\sigma^R_v$.
Note that   the elements in $\sigma^R$   come naturally in pairs indexed by $R$, with each pair giving  the two branches of the corresponding node in $R$.

For each  $e\in R$, consider  the 2-pointed, polarized  family of rational curves,  $(\ma E _e, \sigma_e,\mathcal L_e)$, over $B$ 
 such that
$\ma E _e=\PP^1_k\times B$, the two sections of $\sigma_e$ are constant (say equal to $\{ 0, \infty \}$), and  the polarization is $\mathcal L_e=\pi_{\PP^1_k}^*\cO(1)$.
Now,  each pair in $\sigma^R$ can be glued to the corresponding pair $\sigma_e$ for all $e\in R$. This gives   a gluing of
  $\sqcup_{v\in V(G/P)} \ma Z_v$ with $\sqcup_{e\in R}\ma E_e$, and hence a family of connected curves over $B$, written
  $\widehat{\ma X}\to B$. This family  is endowed with 
marked points $\sigma:=\cup_{v\in V(G/P)} (\sigma_v\smallsetminus \sigma^R_v)$, 
so that the fibers   have  dual graph $\widehat G_R$.
We have a line bundle $\wh{\ma L}$ on $\widehat{\ma X}$ by gluing $ \ma L_v$ with $\ma L_e$, for every $v\in V(G/P)$ and $ e\in R$ in the obvious way (the choice of gluing will not matter).
Then  $(\widehat{\ma X}, \sigma,\wh{\ma L})\in \ma S_{(G,P,s)}(B)$, and we set 
$$
\theta \Big(\prod_{v\in V(G/P)}(\ma Z_v,\sigma_v, \ma L_v)\Bigr):=(\widehat{\ma X}, \sigma,\wh{\ma L}).
$$

We  let  $\eta$ be the composition of $\theta$ with the quotient  map described above,
$$
\eta:\hcS_{(G,P,s)}\la \left[\frac{\hcS_{(G,P,s)}}{\Aut(\oP)}\right]\stackrel{\theta}{\la}\ma S_{(G,P,s)}.
$$

We   want to show that $\eta$ is the quotient of the natural action of $\Aut(G,P,s)$ on $\hcS_{(G,P,s)}$.
 By Lemma~\ref{sequence}   we have that  $\Aut_G(G/P,s)$ acts on $[\frac{\hcS_{(G,P,s)}}{\Aut(\oP)}]$.

From the description in the previous part of the proof, we have that two points have the same image under $\theta$ if and only if they are conjugate by  
$\Aut_G(G/P,s)$. This shows that $\theta$ is the quotient by $\Aut_G(G/P,s)$, and hence $\eta$ is the quotient by  $\Aut(G,P,s)$.
\end{proof}

We have   morphisms  
$\tcS_{(\ol P_v,\ol P_v,s_v)}\ra \wt{\cM}_{\ol P_v}$ whose product gives a morphism
\[
\hcS_{(G,P,s)}\la \prod_{v\in V(G/P)}\wt{\cM}_{\ol P_v}\cong \wt{\cM}_G.
\]
By the proposition
  and the universal property of the fiber product 
we have a morphism $\lambda\col\hcS_{(G,P,s)}\ra \wt{\ma S}_{(G,P,s)}$ and a commutative diagram

 \begin{equation}\label{diag:sprime}
\xymatrix@=.4pc{
&&  \wt{\ma S}_{(G,P,s)} \ar[dd]\ar[rr] &&\wt{\cM}_G \ar[dd]  \\
  \hcS_{(G,P,s)} \ar[drr]^{\eta}\ar[rru]^{\lambda}&&&&&&\\
 &&\ma S_{(G,P,s)} \ar[rr]  &&   \ma{M}_G
}
\end{equation}
\begin{lemma}
  \label{lm:quotient}
 The map
 $\lambda$ in diagram \eqref{diag:sprime} is injective,
 and it is an isomorphism  if and only if  $\Aut(G,P,s)=\Aut(G)$.
\end{lemma}
  
\begin{proof}
  By definition,  $\wt{\cS}_{(G,P,s)}$ parametrizes spin curves whose underlying
 graph is identified with
$G$ and whose spin stucture is of type  $\alpha_*(P,s)$ for some $\alpha\in \Aut(G)$. 

Similarly,  $\hcS_{(G,P,s)}$ parametrizes disjoint unions, indexed by $v\in V(G/P)$, of spin curves
 whose underlying
 graph is identified with
$ \oP_v $ and whose spin stucture is of type  $\alpha^v_*(\oP_v,s_v)$  for some $\alpha^v\in \Aut(\oP_v)$.
 Since $\prod(\Aut (\oP_v))=\Aut(\oP) \subset \Aut(G)$, every such  union determines   a unique element of $\wt{\cS}_{(G,P,s)}$, and is uniquely determined by it.
Hence $\lambda$ is injective.
 
 For   $\alpha\in \Aut(G)$,
write $(P_{\alpha},s_{\alpha})=\alpha_*(P,s)$.
We have  $\ma S_{(G,P_{\alpha},s_{\alpha})}=\ma S_{(G,P,s)}$ (in $\Sgnbst$), hence
$\wt{\cS}_{(G,P,s)}=\wt{\cS}_{(G,P_{\alpha},s_{\alpha})}$. 
We define
$
\lambda_\alpha\col \hcS_{(G,P_\alpha,s_{\alpha})}\la \wt{\cS}_{(G,P,s)}
$
exactly as we did for  $\lambda$  (which is now $\lambda=\lambda_{id}$), so that  
 we have  
$$
\wt{\cS}_{(G,P,s)} =\cup_{\alpha\in \Aut(G)}\lambda_{\alpha}(\hcS_{(G,P_\alpha,s_{\alpha})}).
$$
Now observe  that    $\lambda_{\alpha_1}(\hcS_{(G,P_{\alpha_1},s_{\alpha_1})})= \lambda_{\alpha_2}(\hcS_{(G,P_{\alpha_2},s_{\alpha_2})})$   for   $\alpha_1,\alpha_2\in \Aut(G)$
 if and only if $\alpha_2\alpha^{-1}_1\in \Aut(G,P,s)$,
 and so we are done.
 \end{proof}

 \section{Irreducibility of the strata}\label{sec:irr}
 Our goal here is to  show that $\hcS_{(G, P,s)}$  and    $\cS_{(G, P,s)}$ are irreducible for every stable spin graph  $(G,P,s)$.  In Propositions~\ref{lem:easy} and   \ref{prop:b1}  we concentrate on certain  types of graphs and prove  it  directly. 
We then use this to prove the general statement, by induction,  
in Theorem~\ref{thm:irr}.

 \subsection{Some special cases} We begin with some simple cases.

\begin{prop}\label{lem:easy}
Let $(G,P,s)$ be a stable spin graph of genus $g$.  Then $\hcS_{(G, P,s)}$
is irreducible in the following cases:
\begin{enumerate}[(a)]
\item
\label{leasysm} $E =\emptyset$;
\item  
\label{leasysm1}
$g\leq 1$;
\item 
\label{lem:easyreg}$G$ is  $3$-regular (hence weightless).
\end{enumerate}
\end{prop}
 
\begin{proof}
In \eqref{leasysm} we are dealing  with smooth curves  and $\hcS_{(G,P,s)}=\ma S_{(G,G,s)}$. The claim follows is \cite[Lm. 6.3]{cornalba} if $n=0$, and  \cite[Thm. 3.3.1]{J00} in general.


Let $\{(\ol P_v,s_v),v\in V(G/P)\}$ be 
the connected components of $(G,P,s)$; by  \eqref{eq:Sprimedef}  it is enough to  show that $\tcS_{(\ol P_v,\ol P_v,s_v)}$ is irreducible.

Assume $g\le1$.  Then either 
 $E(\ol P_v)=\emptyset$ and hence $\tcS_{(\ol P_v,\ol P_v,s_v)}$ is irreducible by \eqref{leasysm}, or $\ol P_v$ is a cycle with all vertices of weight $0$. So, we can assume that $G$ is a cycle  with all vertices of weight $0$ and $P=G$, hence 
 $\hcS_{(G, G,s)}=\ma S_{(G,G,s)}\times_{\cM_G}\wt{\cM}_G$. Let  $X$ 
   have $G$ as dual graph. After fixing the sign $s(v)$ of the unique vertex $v\in V(G/G)$, there is exactly one theta characteristic $L$ on $X$ of type $(G,s)$,   given by taking the trivial bundle on each component of $X$ and choosing the  gluing so that  $h^0(X,L)\equiv s(v)\text{ mod } (2)$ (see Fact \ref{CCK60} (b)). We get $\ma S_{(G,G,s)}\cong \cM_G$, and hence $\hcS_{(G,G,s)}\cong \wt{\cM}_G$, which is irreducible.
 
If $G$ is $3$-regular each $\ol P_v$ has genus at most $1$, hence  $\tcS_{(\ol P_v,\ol P_v,s_v)}$ is irreducible  by  \eqref{leasysm1}, so $\hcS_{(G, P,s)}$ is irreducible.
\end{proof}

For a vertex $v$ of $G$ we let  $\operatorname{loop}(v)$ be the number of loops  based at $v$ and 
$$g(v):=\operatorname{loop}(v)+w(v),$$
so that for any curve $X$  having $G$ as dual graph
the irreducible component $C_v$ of $X$ has arithmetic genus equal to $g(v)$.

\begin{definition}
A stable   graph  $G$ of genus $g\geq 2$ with   $E \ne \emptyset$ is \emph{basic} if it is cyclic, and  if
   for every $v\in V$ we have  $w(v)\leq 1$  and the following holds
\begin{equation}\label{eq:basic}
w(v)+\deg (v)+\ell(v)\le 4,
\end{equation}
with equality  only if $\operatorname{loop}(v)\geq 1$.

If $G$ is basic, we say that $(G,G,s)$ is a {\it basic spin graph} for any $s$.
\end{definition}

\begin{remark}
\label{rem:basic}
 Let $G$ be a basic   graph and  $v\in V $. Then 
\begin{enumerate}[(a)]
 \item
 $2\leq \deg (v) \leq 4$ and if $\deg(v)=4$ then $w(v)=\ell(v)=0$.
 \item
 If $|V|=1$ then $g=2$.
\item
If $|V|\geq 2$ then  
the graph obtained by removing from $G$ every loop   is    a cycle.
  \item
  $g(v)\leq 2$ and if equality holds then $|V|=1$.
\end{enumerate}
\end{remark}
Let $G$ be a basic graph. For $0\le i\leq 1$ and $0\le j\le2$ we set
\begin{equation}\label{eq:Vij}
V_{i,j} :=\{v\in V : w(v)=i \text{ and } g(v)=j\}.
\end{equation}
By the above remark, every vertex of $G$  is contained in some $V_{i,j}$. Set  
$
V^+:=V \smallsetminus V_{0,0} .
$
 The complete list of types of vertices  
is in Figure \ref{fig:Vij}.  
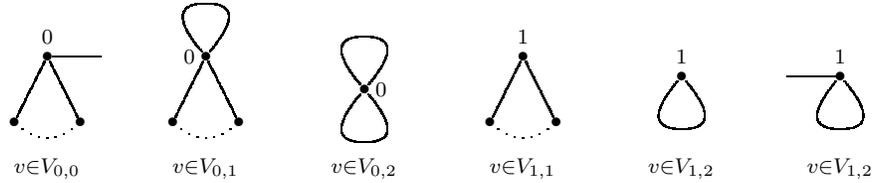
\begin{figure}[h]
\[
\begin{xy} <25pt,0pt>:
(-6,1)*{\scriptstyle\bullet}="a";
"a"+(0.5,-1)*{\scriptstyle\bullet}="b";
"a"+(-0.5,-1)*{\scriptstyle\bullet}="c";
"a";"b"**\crv{"b"};
"a";"c"**\crv{"c"};
{\ar@/_/@{..}(-6.5,0)*{};(-5.5,0)*{}};
(-3.6,1)*{\scriptstyle\bullet}="d";
"d"+(0.5,-1)*{\scriptstyle\bullet}="e";
"d"+(-0.5,-1)*{\scriptstyle\bullet}="f";
"d";"e"**\crv{"e"};
"d";"f"**\crv{"f"};
{\ar@/^/@{..}(-3.1,0)*{};(-4.1,0)*{}};
"d"+(0,0.8);"d"**\crv{"d"+(-0.7,0.8)};
"d"+(0,0.8);"d"**\crv{"d"+(0.7,0.8)};
(-1.2,0.5)*{\scriptstyle\bullet}="g";
"g"+(0,0.8);"g"**\crv{"g"+(-0.7,0.8)};
"g"+(0,0.8);"g"**\crv{"g"+(0.7,0.8)};
"g"+(0,-0.8);"g"**\crv{"g"+(-0.7,-0.8)};
"g"+(0,-0.8);"g"**\crv{"g"+(0.7,-0.8)};
(1.2,1)*{\scriptstyle\bullet}="h";
"h"+(0.5,-1)*{\scriptstyle\bullet}="i";
"h"+(-0.5,-1)*{\scriptstyle\bullet}="l";
"h";"i"**\crv{"i"};
"h";"l"**\crv{"l"};
{\ar@/^/@{..}(1.7,0)*{};(0.7,0)*{}};
(3.6,0.7)*{\scriptstyle\bullet}="m";
"m"+(0,-0.8);"m"**\crv{"m"+(-0.7,-0.8)};
"m"+(0,-0.8);"m"**\crv{"m"+(0.7,-0.8)};
(6,0.7)*{\scriptstyle\bullet}="n";
"n"+(0,-0.8);"n"**\crv{"n"+(-0.7,-0.8)};
"n"+(0,-0.8);"n"**\crv{"n"+(0.7,-0.8)};
"a"+(0,0.3)*{\scriptstyle{0}};
"a"+(0,-1.7)*{\scriptstyle{v\in V_{0,0}}};
"d"+(-0.25,0)*{\scriptstyle{0}};
"d"+(0,-1.7)*{\scriptstyle{v\in V_{0,1}}};
"g"+(0.25,0)*{\scriptstyle{0}};
"g"+(0,-1.2)*{\scriptstyle{v\in V_{0,2}}};
"h"+(0,0.3)*{\scriptstyle{1}};
"h"+(0,-1.7)*{\scriptstyle{v\in V_{1,1}}};
"m"+(0,0.3)*{\scriptstyle{1}};
"n"+(0,0.3)*{\scriptstyle{1}};
"m"+(0,-1.4)*{\scriptstyle{v\in V_{1,2}}};
"n"+(0,-1.4)*{\scriptstyle{v\in V_{1,2}}};
"a"+(0.8,0);"a"**\crv{"a"};
"n"+(-0.8,0);"n"**\crv{"n"};
\end{xy}
\]
\caption{Possible types of vertices of a basic   graph.}
\label{fig:Vij}
\end{figure}

Let $X$ have $G$ as dual graph.
For $v\in V^+$, we let $\nu\col C_v^\nu\ra C_v$ be the normalization of $C_v$. We let $\rho_{C_v} \col C_v^\nu\ra \mb P^1_k$ be the degree-two   map induced by the $g^1_2$ of $X=C_v$ if $v\in V_{0,2}\cup V_{1,2}$, or by the linear system $|p_v+p'_v|$ of $C_v$ if  $v\in V_{0,1}\cup V_{1,1}$, where  $\{p_v,p'_v\}:=C_v\cap \ol{X\setminus C_v}$.
Let $R(\rho_{C_v})$ be the 
image via $\nu$ of the
ramification divisor of $\rho_{C_v}$ and write 
\[
 R(\rho_{C_v})=
\begin{cases}
\begin{array}{ll}
r_{v,1}+r_{v,2}, &  \text{ if } 
w(v)=0;\\
r_{v,1}+r_{v,2}+r_{v,3}+r_{v,4}, & \text{ if } 
w(v)=1.
\end{array}
\end{cases}
\]
Each $r_{v,i}$  is a smooth point  of $X$, hence    $ R(\rho_{C_v})$ is a Cartier divisor on $X$.

A \emph{$G$-collection} is a collection of indices $I=(i_v)_{v\in V^+}$, with $i_v\in I_v$ for every $v\in V^+$, where \begin{equation}\label{eq:Iv}
 I_v:=
\begin{cases}
\begin{array}{ll}
\{1,2\}, & \text{ if } 
w(v)=0;\\
\{1,2,3,4\}, & \text{ if } 
w(v)=1.
\end{array}
\end{cases}
\end{equation}

\begin{lemma}\label{lem:thetaform}
Let $G$ be a basic     graph and $X$ a  curve with $G$ as dual graph. Fix $u\in V^+$ and set $\epsilon_u=(-1)^{w(u)}$.
 Then the  odd and even theta characteristics on $X$ are, respectively, the following line bundles, $L_{X,I}$ and $M_{X,I}$
$$
L_{X,I}=\ma O_X\Bigr(\sum_{v\in V^+} r_{v,i_v}\Bigl), 
\quad  
M_{X,I}=\ma O_X\Bigr(\epsilon_u 3r_{u,i_{u}}-\epsilon_u  R(\rho_{C_{u}})+\sum_{v\in V^+\smallsetminus\{u\}} r_{v,i_v}\Bigl)
$$
for every $G$-collection $I=(i_v)_{v\in V^+}$. 

Moreover, we have $h^0(X,L_{X,I})=1$ and $h^0(X,M_{X,I})=0$ for every  $I$.
\end{lemma}
 
\begin{proof}

Assume   $V=\{u\}$, hence   $G$ has genus 2. We have to check that the odd and even theta characteristics of $X$ are, respectively, the line bundles $\ma O_X(r_{u,i})$  and $\ma O_X(\epsilon_u 3r_{u,i}-\epsilon_u \nu_*R(\rho_{u}))$ for $i\in I_{u}$.   Since $\omega_X\cong \ma O_X(2r_{u,i})$ for every $i$, this is an easy checking.

Assume now   $|V|\ge2$, hence $V_{0,2}=V_{1,2}=\emptyset$.  Since the curve $X$ is fixed, we     write $L_I=L_{X,I}$ and $M_I=M_{X,I}$.
For different $G$-collections $I,J$, the line bundles $L_{I}\otimes L_{J}^{-1}$ and $M_I\otimes M_J^{-1}$ are not trivial, since so are their restrictions to at least one component of $X$. 
As the number of $G$-collections is  $2^{|V_{0,1}|}4^{|V_{1,1}|}$,  this is also the number of    line bundles of type $L_I$, and    of line bundles of type $M_I$, which is exactly the  number of odd and even theta characteristics on $X$, by Fact~\ref{CCK60}.

Let $E'\subset E$ be the set edges which are not loops. 
By Remark~\ref{rem:basic} there exists an 
 $e\in E'$, and $G-e$  becomes a tree after all loops are removed. Therefore
  the partial normalization, $\nu\col X^\nu_e\to X$,   of $X$ at   $e$  is tree-like. 

For every $G$-collection $I=(i_v)_{v\in V^+}$, the dualizing sheaf of $X$ is
\begin{equation}
 \label{dualX}
\omega_X\cong \ma O_X\left(\sum_{v\in V^+} 2r_{v,i_v}\right),
\end{equation}
as $h^0(X_e^\nu,\nu^*\ma O_X(\sum_{v\in V^+} 2r_{v,i_v}))=g$. Hence $h^0(X,\ma O_X(\sum_{v\in V^+} 2r_{v,i_v}))=g$, by \cite[Lm. 2.2.4(2)]{Cap09}, and  hence $L_I$ is a theta characteristic for every $I$. 

Let $N$ be a theta characteristic of $X$.   
We have $\deg N_{|C_v}\leq 1$ and hence $h^0(C_v, N_{|C_v})\le 1$ for every $v \in V$. Therefore
$$
h^0(X,N)\le h^0(X^\nu_e,\nu^*N)\leq \sum_{v\in V(G)}h^0(C_v,N_{|C_v})-|E' \smallsetminus\{e\}|\leq 1.$$ 


(We use $|V |=|E' |$.)
So $h^0(X,N)=1$ if $N$ is odd, and $h^0(X,N)=0$ if $N$ is even. In particular, since of course $h^0(X,L_I)>0$, the odd theta characteristics on $X$ are as stated.
There remains to show every  $M_I$  is an even theta characteristic.
By \cite[Thm. 2.14]{Har82} this is true if we find a $G$-collection $J$ such that $M_I\otimes L_J^{-1}$ is a nontrivial square root of $\ma O_X$. Define $J=(j_v)_{v\in V^+}$ by choosing $j_{u}=i_{u}$ if $u\in V_{0,1}$, and $j_{u}\ne i_{u}$ if $u\in V_{1,1}$, and setting $j_v=i_v$ for every $v\in V^+\ssm \{u\}$. Then $M_I\otimes L_J^{-1}$ is non trivial,  as
one easily checks
\[
(M_I\otimes L_J^{-1})|_{C_{u}}
\cong 
\begin{cases}
\begin{array}{ll}
\ma O_{C_{u}}(r_{u,i_{u}}-r_{u,k}), & \text{ if }  \{i_{u},k\}=\{1,2\},\\
\ma O_{C_{u}}(r_{u,k}-r_{u,k'}), &  \text{ if }  \{i_{u},j_{u},k,k'\}=\{1,2,3,4\},
\end{array}
\end{cases}
\]
which is nontrivial. Finally we show that $T:=(M_I\otimes L_J^{-1})^{\otimes 2}$ is trivial. By the definition of $M_I$ and $L_J$, we have 
\[
T\cong 
\begin{cases}
\begin{array}{ll}
\ma O_X(2r_{u,i_{u}}-2r_{u,k}), &  \text{ if }  \{i_{u},k\}=\{1,2\}; \\ 
\ma O_X(-4r_{u,i_{u}}+2r_{u,k}+2r_{u,k'}), &   \text{ if } \{i_{u},j_{u},k,k'\}=\{1,2,3,4\}.
\end{array}
\end{cases}
\]  
An easy calculation shows that  $h^0(X^\nu_e,\nu^*\ma O_X(2r_{u, i_{u}}))=2$ which, again by \cite[Lm. 2.2.4(2)]{Cap09}, implies that $h^0(X,\ma O_X(2r_{u, i_{u}}))=2$. It follows that 
 $\ma O_X(2r_{u, i_{u}})\cong \ma O_X(2r_{u,h})$ for $h\in\{k,k'\}$, and hence $T$ is trivial.
 \end{proof}

\begin{remark}\label{def:rhoz}
Let $z\in \cM_{0,4}$   represent  the pointed curve $(\mathbb P^1;p_1,p'_1,p_2,p'_2)$.  
Then $z$ determines a      degree-2   map,  $\rho_z\col \mathbb P^1 \ra \mathbb P^1$,
such that $\rho_z(p_i)=\rho_z(p'_i)$ for $i=1,2$. We denote by $R(\rho_z)=r_{z,1}+r_{z,2}$   the ramification divisor of $\rho_z$. 
We have an isomorphism
 $(\mathbb P^1;p_1,p'_1,p_2,p'_2)\cong (\mathbb P^1;p_1',p_1,p'_2,p_2) 
$ 
induced by the involution associated to  $\rho_z$.
Therefore the involution on $\cM_{0,4}$ exchanging the first two marked  points 
acts like the involution exchanging the last two points.   We denote this involution by $\alpha$, so that 
the quotient map, $\cM_v\to  \cM_{0,4}/\langle\alpha\rangle$,   is   two-to-one  and ramifies only at
$(\mathbb P^1;  0,\infty,1,-1)$.
\end{remark}

\begin{proposition}\label{prop:b1}
Let $(G,G,s)$ be a basic stable spin graph. Then $\hcS_{(G,G,s)}$  is irreducible.
\end{proposition}

\begin{proof}  
For any  $v\in V_{0,0}$ we have $\cM_v=\cM_{0,3}$ which is a point. Hence the natural \'etale map from 
$\hcS_{(G,G,s)}=\tcS_{(G,G,s)}$ to $\wt{\cM}_G$ can be written 
 as follows
$$
\xi\col \tcS_{(G,G,s)}\la \prod_{v\in  V^+}\cM_v= \wt{\cM}_G
$$

Since   $\tcS_{(G,G,s)}$ is smooth it suffices to prove it is connected.  
We assume throughout the proof that $\text{char}(k)=0$: the proof in positive characteristic will follow using \cite[Thm. 4.17]{DM69}, as in \cite[Thm 3.3.11]{J00}.  We also assume that $(G,G,s)$ is odd; the proof for the other case   is  the  same. 
By Lemma \ref{lem:thetaform} we can write 
\[
\xi^{-1}(z)=\{ L_{X,I} : \  \forall \text{ $G$-collections }     I\}
\]
where $X$ is the   curve parametrized by the image of $z$ in $\cM_G$.
To show that $\tcS_{(G,G,s)}$ is connected, it suffices to find a point $z\in \wt{\cM}_G$ such that  for every $G$-collections $I,J$ there is a path in $\tcS_{(G,G,s)}$ connecting $ L_{X,I}  $ to $ L_{X,J}$. 

 Write $I=(i_v)_{v\in  V^+}$ and $J=(j_v)_{v\in  V^+}$, where $i_v,j_v\in I_v$. Let $\alpha_v$ be  the involution of $I_v$    switching $i_v$ and $j_v$ for every $v\in  V^+$.

First consider a vertex $v\in V_{0,1}\cup V_{0,2}$, so that  $\cM_v=\cM_{0,4}$.

Let $z_v\in \cM_v $ correspond to    $(\mathbb P^1; 0,\infty,1,-1)$, let $\rho=\rho_{z_v}$ be the degree-2   map of Remark~\ref{def:rhoz}, and let 
$r_{z_v,1}$ and $r_{z_v,2}$ be its   two ramification points.
This   enables us to define  the following two points in $\cM_{0,4}$ 
$$
(\mathbb P^1; \rho (r_{z_v,1}),\rho  (r_{z_v,2}),\rho (0),\rho (1)),
\quad
 \quad
(\mathbb P^1; \rho (r_{z_v,\alpha_v(1)}),\rho (r_{z_v,\alpha_v(2)}),\rho (0),\rho (1)).
$$
Fix a  path, $\psi_v$, in $\cM_{0,4}$ between these two points. 

Now, to any   point  $(\mathbb P^1;q_1,q_2,q_3,q_4)$ of  $\cM_{0,4}$ we associate   the
degree-2 covering of $\mathbb P^1$ branched at  $q_1$ and $q_2$,  and marked by the (four)  pre-images of $q_3$ and $q_4$.
This does not  quite give a point in  $\cM_{0,4}$, because
the four   marked points are not  naturally ordered.  
 But, using   Remark~\ref{def:rhoz}, it clearly  determines    a point in
 $\cM_{0,4}/\langle\alpha\rangle$. 
 In this way the 
  path  $\psi_v$  gives rise to a closed path, $\psi'_v$, in  $\cM_{0,4}/\langle\alpha\rangle$ based at $[(\mathbb P^1;  0,\infty,1,-1)]$. 
Now,  there exists a   lifting   of $\psi'_v$ to $\cM_v$, written
 $\kappa_v$,   and it is clear    that   $\kappa_v$  is a closed path based at $z_v$. 

Now consider a vertex $v\in V_{1,1}\cup V_{1,2}$. Let $z_v\in \cM_v=\cM_{1,2}$ be the point parametrizing  a pointed curve $(Z_v;p_v,p'_v)$. 
 Let $\rho_{z_v}\col Z\ra \mathbb P^1$ be    the degree-2 map   induced by the linear system $|p_v+p_v|$, and    $R(\rho_{z_v})=r_{z_v,1}+r_{z_v,2}+r_{z_v,3}+r_{z_v,4}$ be its ramification divisor.
Write $\rho= \rho_{z_v}$ for simplicity.
  Consider a   path $\psi_v$ in $\cM_{0,5}$ starting from the point 
\[
(\mathbb P^1; \rho (r_{z_v,1}),\rho (r_{z_v,2}), \rho (r_{z_v,3}), \rho (r_{z_v,4}),\rho (p_v))
\]
and ending at the point 
\[
(\mathbb P^1; \rho (r_{z_v,\alpha_v(1)}),\rho (r_{z_v,\alpha_v(2)}),\rho (r_{z_v,\alpha_v(3)}),\rho (r_{z_v,\alpha_v(4)}),\rho (p_v)).
\]
Now,  
we have a morphism $\cM_{0,5}\to \cM_{1,2}$, mapping $(\mathbb P^1;q_1,q_2,q_3,q_4,q_5)$
to $(Z_v;p,p')$, defined as  the pointed curve endowed with the degree-2 covering  $Z_v\to \PP^1$ branched at the first four points and marked by the two   pre-images, $p,p'$, of $q_5$,
taken in any order (we have an isomorphism $(Z_v;p,p')\cong (Z_v;p',p)$).
The image of $\psi_v$ under this morphism  gives rise to a   closed path, $\kappa_v$, in $\cM_{1,2}$ based at $z_v$.

Now take a point $z=\prod_{v\in  V^+}z_v\in \wt{\cM}_G$ with $z_v\in \cM_v$   as above. Take the   closed path $\kappa$ in $\wt{\cM}_G$ based at $z$ given by the product of the closed paths $\kappa_v$ in each $\cM_v$.  
For every $v\in V^+$, the map $\rho_{z_v}$  coincides with the map $\rho_{C_{v}}\col C_{v}^{\nu}\ra \mathbb P^1$ defined before Lemma~\ref{lem:thetaform}, and hence we have  $R(\rho_{z_v})=R(\rho_{C_{z,v}})$.  
By the description of the theta characteristics of $X$ in Lemma \ref{lem:thetaform} in terms of the ramification points of the maps $\rho_{C_{z,v}}$,   we see that,  by construction,  there is a lifting to  $\tcS_{(G,G,s)}$ of the path $\kappa$ to a  path    starting from $ L_{X,I} $ and ending at  $ L_{X,J}$, as wanted. 
\end{proof}

 \subsection{Proof of the irreducibility} To prove Theorem~\ref{thm:irr} we will use the next lemma to handle
 graphs not treated in Propositions~\ref{lem:easy} and   \ref{prop:b1}.
\begin{lemma}\label{lem:specia}
Let $G$ be a non basic Eulerian stable graph of genus $g\ge 2$ with $E(G)\ne \emptyset$. Consider a spin structure of type $(G, s)$ on $G$. 
Then there is a stable graph $G'$ contracting to $G$, and a connected spin structure $(P',s')$ on $G'$ such that:
\begin{enumerate}[(a)]
\item $|E(G')|=|E(G)|+1$ and  $b_1(G')=b_1(G)$;
\item $\Aut(G',P',s')=\Aut(G')$;
\item  $\gamma_*(P', s')=(G, s)$ for every contraction $\gamma\col G'\ra G$;
\item
any spin structure  on $G'$  contracting to $(G,G,s)$ is equal to $(G',P',s')$.
\end{enumerate}\end{lemma}

\begin{proof}
Since $G$ is not basic, there exists a vertex $v$ of $G$ such that 
\begin{equation}
 \label{nobasic}
w(v)+\deg (v)+\ell(v)\geq 4,
\end{equation}
and if equality holds then $G$ has    no loop based at $v$.

Let $C\subset G$ be a cycle containg $v$. Denote by $e_1, e_2$ the edges of $C$ adjacent at $v$
and denote by $E^*_v$ the set of edges of $G$ adjacent to $v$ and different from $e_1$ and $e_2$. Then $E^*_v$ has even (possibly zero) cardinality.
The graph $G'$ will be defined as a blow-up of $v$ separating $e_1, e_2$ in such a way that $G'$ is stable.
We shall denote by $e'\in E(G')$ the edge to be contracted to $v$, and  by $u_1,u_2 \in V(G')$ the end points of $e'$.
Note that $u_1$ and $u_2$ have degree at least $2$.
The weight function, $w'$, of   $G'$ must  be such that
$w'(u_1)+w'(u_2)=w (v)$. Similarly, the number of legs at each vertex must satisfy
  $\ell'(u_1)+\ell'(u_2)=\ell (v)$.
 
 We   need to distribute the edges in $E^*_v$  between $u_1$ and $u_2$, and  define $w'$ and $\ell'$  so that  $G'$ is stable.
It is clear that \eqref{nobasic} implies that this is always possible.
On the other hand it is not always possible to have $G'$ also Eulerian.
Indeed,  $G'$ is Eulerian if and only if  $u_1$ and $u_2$ have even degree. For that to be possible we need either
  $\deg(v)\geq 6$ (i.e.  $|E^*_v|\geq 4$ and we can attach two edges at $u_1$ and the remaing ones at $u_2$)
 or $\deg(v)= 4$ and  $w(v)+\ell(v)\geq 1$
   (so that we can attach the two edges of $E^*_v$ at $u_1$ and define $w'$ and $\ell'$ so that $w'(u_2)+\ell'(u_2)=1$),
   or else $\deg(v)=2$ (so that  $w(v)+\ell(v)\geq 2$ and   we define $w'$ and $\ell'$ so that
    $w'(u_i)+\ell'(u_i)\geq 1$ for $i=1,2$).

The only   case left is $\deg(v)=4$  and  $w(v)+\ell(v)=0$, which is different in that  $G'$ is not Eulerian, so we treat it  at the end.
    
In  all    other  cases we have   a  Eulerian stable graph $G'$ such that $G'/e'=G$; see Figure~\ref{fig:deg6}.  Let $(G',G',s')$  be  such that the value of $s'$ (on the unique vertex of  $V(G'/G')$)
is equal to the value of $s$ (on the unique vertex of $V(G/G)$).  Then it is clear that  $(G',G',s')$ satisfies   the statement.

\begin{figure}[h]
\[
\begin{xy} <25pt,0pt>:
(2.5,0.5)*{\scriptstyle\bullet}="b";
(4,0)*{\scriptstyle}="d";
(5,0)*{\scriptstyle}="a";
"b"+0;"b"+(0,0.8)**\crv{"b"+(0,0)};
"b"+0;"b"+(0,-0.8)**\crv{"b"+(0,0)};
"b"+0;"b"+(1,0.5)**\crv{"b"+(0,0)};
"b"+0;"b"+(-1,0.5)**\crv{"b"+(0,0)};
"b"+0;"b"+(1,-0.5)**\crv{"b"+(0,0)};
"b"+0;"b"+(-1,-0.5)**\crv{"b"+(0,0)};
{\ar@{->}(5.8,0.5)*{};(4.5,0.5)*{}};
{\ar@/_/@{..}(7,1)*{};(7,0)*{}};
{\ar@/^/@{..}"b"+(0,0.8)*{};"b"+(1,0.5)*{}};
{\ar@/_/@{..}"b"+(0,-0.8)*{};"b"+(1,-0.5)*{}};
{\ar@/_/@{..}(1.5,1)*{};(1.5,0)*{}};
(8,1)*{\scriptstyle\bullet}="e";
(8,0)*{\scriptstyle\bullet}="g";
(9,0)*{\scriptstyle}="h";
"g"+(-1,0);"g"+(0,0)**\crv{"g"};
"g"+0;"e"+0**\crv{"g"+(0,0.1)};
"g"+(0,0);"g"+(0,-0.8)**\crv{"g"+(0,0)};
"g"+(0,1);"g"+(0,1.8)**\crv{"g"+(0,1)};
"g"+(0,1);"g"+(1,1.5)**\crv{"g"+(0,1)};
"g"+(0,0);"g"+(1,-0.5)**\crv{"g"+(0,0)};
{\ar@/^/@{..}"g"+(0,1.8)*{};"g"+(1,1.5)*{}};
{\ar@/_/@{..}"g"+(0,-0.8)*{};"g"+(1,-0.5)*{}};
"e"+(-1,0);"e"+(0,0)**\crv{"e"+(0,0)};
"a"+(2.3,1.15)*{\scriptstyle{e'_2}};
"a"+(2.3,-0.2)*{\scriptstyle{e'_1}};
"a"+(3.2,0.55)*{\scriptstyle{e'}};
"a"+(2.8,0.2)*{\scriptstyle{u}_1};
"a"+(2.8,0.8)*{\scriptstyle{u_2}};
"b"+(-0.7,0.2)*{\scriptstyle{{e}_2}}; 
"b"+(-0.7,-0.55)*{\scriptstyle{e_1}};
"b"+(-0.1,0.22)*{\scriptstyle{v}};
"a"+(-4.5,0.5)*{G};
"a"+(5,0.5)*{G'};
\end{xy}
\]
\caption{ }
\label{fig:deg6}
\end{figure}
We are left with the case  $\deg(v)=4$ and  $w(v)+\ell(v)=0$;
then $G$ has no loops at $v$. Now we define $G'$ and the corresponding  contraction, $\gamma$,  as in 
Figure \ref{fig:deg4noloop}.  
Notice that $G'$ is stable, and has exactly two vertices of odd degree, namely $u_1$ and $u_2$.
Hence 
  $G'$ is   not Eulerian, but its subgraph 
  $P':= G' -e'$, is Eulerian. 
  Consider the 
 spin structure   $( P',s')$ where  $s'$ has the same value as  $s$. Of course, $\gamma_*(P',s')=(G,s)$
 and it obvious that $P'$ is the only cyclic subgraph of $G'$ such that $\gamma_*P'=G$.
 
 To prove that    $\Aut(G',P',s')=\Aut(G')$ it suffices to observe that every automorphism of $G'$ must 
 leave $\{u_1,u_2\}$ invariant, and hence it must
 fix $e'$
which  is the only edge containing $u_1$ and $u_2$ ($G$ has no loop at $v$).

Now, any contraction of $G'$ to $G$ must take $u_1$ and $u_2$ to vertices of even degree. Therefore the contraction $\gamma$ is the only
possible one.  \end{proof}
 \begin{figure}[h]
\[
\begin{xy} <25pt,0pt>:
(2.5,0.5)*{\scriptstyle\bullet}="b";
(1.25,0.5)*{\scriptstyle\bullet};
(3.75,0.5)*{\scriptstyle\bullet};
(6.75,0.5)*{\scriptstyle\bullet};
(9.25,0.5)*{\scriptstyle\bullet};
(4,0)*{\scriptstyle}="d";
(5,0)*{\scriptstyle}="a";
"b"+0;"b"+(1,0.5)**\crv{"b"+(0,0)};
"b"+0;"b"+(-1,0.5)**\crv{"b"+(0,0)};
"b"+0;"b"+(1,-0.5)**\crv{"b"+(0,0)};
"b"+0;"b"+(-1,-0.5)**\crv{"b"+(0,0)};
{\ar@{->}(5.8,0.5)*{};(4.5,0.5)*{}};
{\ar@/^/@{..}(9,1)*{};(9,0)*{}};
{\ar@/_/@{..}(7,1)*{};(7,0)*{}};
{\ar@/^/@{..}(3.5,1)*{};(3.5,0)*{}};
{\ar@/_/@{..}(1.5,1)*{};(1.5,0)*{}};
(8,1)*{\scriptstyle\bullet}="e";
(8,0)*{\scriptstyle\bullet}="g";
(9,0)*{\scriptstyle}="h";
"g"+(-1,0);"g"+(1,0)**\crv{"h"};
"g"+0;"e"+0**\crv{"g"+(0,0.1)};
"e"+(-1,0);"e"+(1,0)**\crv{"e"+(0,0)};
"a"+(3.7,1.2)*{\scriptstyle{e_3}};
"a"+(2.3,1.2)*{\scriptstyle{e_2}};
"a"+(3.7,-0.2)*{\scriptstyle{e_4}};
"a"+(2.3,-0.2)*{\scriptstyle{e_1}};
"a"+(3.25,0.55)*{\scriptstyle{e'}};
"a"+(3,-0.25)*{\scriptstyle{u_1}};
"a"+(3,1.3)*{\scriptstyle{u_2}};
"b"+(-0.5,0.5)*{\scriptstyle{e_2}}; 
"b"+(0.5,-0.45)*{\scriptstyle{e_4}};
"b"+(-0.5,-0.45)*{\scriptstyle{e_1}};
"b"+(0.5,0.5)*{\scriptstyle{e_3}};
"b"+(0,0.25)*{\scriptstyle{v}};
"a"+(-4.5,0.5)*{G};
"a"+(5,0.5)*{G'};
\end{xy}
\]
\caption{ }
\label{fig:deg4noloop}
\end{figure}

We are ready to prove that $\hcS_{(G, P,s)}$ is irreducible for every stable spin graph $(G,P,s)$.  
We will apply the discussion of Subsection~\ref{rem:defospin}.
First, we illustrate the strategy of the  proof  in the following example, where we show that $\cS_{(G,G,s)}$ is irreducible in case $\ov {\cM}_G$ is smooth.

\begin{example}
Let $G'\ra G$ be the contraction of a stable graph of genus $2$ with $2$ legs as in Figure \ref{fig:workexa}. Notice that the contraction of the edge $e'$ of $G'$ is the unique contraction of $G'$ to $G$. As $G'$ is the unique stable  graph contracting  to $G$, we have  that $\ov{\cM}_G$ is smooth. 
Consider a spin graph of type $(G,G,s)$, and let $(P',s')$ be the spin structure on $G'$ such that $ P'  = [G'- e']$ and   that $s'(v')=s(v)$, where $V(G/G)=\{v\}$ and $V(G'/P')=\{v'\}$.
\begin{figure}[h]
\[
\begin{xy} <25pt,0pt>:
(2.5,0.5)*{\scriptstyle\bullet}="b";
"b"+(2,0)*{\scriptstyle\bullet}="e";
"b"+(-2,0)*{\scriptstyle\bullet}="f";
(4,0)*{\scriptstyle}="d";
(5,0)*{\scriptstyle}="a";
"b"+0;"b"+(2,0)**\crv{"b"+(1,1)};
"b"+0;"b"+(-2,0)**\crv{"b"+(-1,-1)};
"b"+0;"b"+(2,0)**\crv{"b"+(1,-1)};
"b"+0;"b"+(-2,0)**\crv{"b"+(-1,1)};
{\ar@{->}(5.8,0.5)*{};(5.1,0.5)*{}};
"e"+0;"e"+(0,0.6)**\crv{"e"};
"f"+0;"f"+(0,0.6)**\crv{"f"};
(8,1)*{\scriptstyle\bullet}="e";
(8,0)*{\scriptstyle\bullet}="g";
(9,0)*{\scriptstyle}="h";
"g"+(1.5,0.5)*{\scriptstyle\bullet};
"g"+(-1.5,0.5)*{\scriptstyle\bullet};
"g";"g"+(1.5,0.5)**\crv{"g"};
"g";"g"+(-1.5,0.5)**\crv{"g"};
"g"+(0,1);"g"+(1.5,0.5)**\crv{"g"+(0,1)};
"g"+(0,1);"g"+(-1.5,0.5)**\crv{"g"+(0,1)};
"e"+(1.5,-0.5);"e"+(1.5,0.1)**\crv{"e"+(1.5,-0.5)};
"e"+(-1.5,-0.5);"e"+(-1.5,0.1)**\crv{"e"+(-1.5,-0.5)};
"g"+0;"e"+0**\crv{"g"+(0,0.1)};
"a"+(3.7,1)*{\scriptstyle{e_5}};
"a"+(2.3,1)*{\scriptstyle{e_4}};
"a"+(3.7,0)*{\scriptstyle{e_3}};
"a"+(2.3,0)*{\scriptstyle{e_2}};
"a"+(3.25,0.55)*{\scriptstyle{e'}};
"a"+(3,-0.25)*{\scriptstyle{0}};
"a"+(3,1.2)*{\scriptstyle{0}};
"b"+(-1,0.7)*{\scriptstyle{e_4}}; 
"b"+(1,-0.7)*{\scriptstyle{e_3}};
"b"+(-1,-0.7)*{\scriptstyle{e_2}};
"b"+(1,0.7)*{\scriptstyle{e_5}};
"b"+(0,0.25)*{\scriptstyle{0}};
"b"+(-2.2,0)*{\scriptstyle{0}};
"b"+(2.2,0)*{\scriptstyle{0}};
"b"+(7.2,0)*{\scriptstyle{0}};
"b"+(3.8,0)*{\scriptstyle{0}};
"a"+(-5.3,0.55)*{G};
"a"+(5.4,0.58)*{G'};
\end{xy}
\]
\caption{The contraction $G'\ra G$.}
\label{fig:workexa}
\end{figure} 
Now, $\ma S_{(G',P',s')}$ is irreducible, as
$G'$ is $3$-regular (Proposition~\ref{lem:easy}).
We claim
\begin{itemize}
\item[(1)] 
$\ma S_{(G',P',s')}=\ol{\cS}_{(G,G,s)}\cap \pi^{-1}(\cM_{G'})$;
\item[(2)]
 $\ol{\cS}_{(G,G,s)}$ is smooth.
\end{itemize}
Indeed (1) follows by Proposition \ref{thm-strata} and  Lemma \ref{lem:specia}. 
For (2),  recall the notations of Subsection~\ref{rem:defospin} and consider a point $y\in \ma S_{(G',P',s')}$. The homomorphism $\pi^\#\col\wh{\ma O}_{\Mgnbst,\pi(y)}\ra \wh{\ma O}_{\Sgnbst,y}$ is given by $\pi^\#(t_1)=s_1^2$ and $\pi^\#(t_i)=s_i$ for $i\ge 2$, hence we have 
\[
\wh{\ma O}_{\ol{\ma S}_G, y}\cong \frac{k[[t_1,\dots, t_5]]}{(t_2,t_3,t_4,t_5)}\otimes_{k[[t_1,\dots,t_5]]}k[[s_1,\dots,s_5]]\cong\frac{k[[s_1,\dots, s_5]]}{(s_2,s_3,s_4,s_5)}.
\]
Therefore $\ol{\cS}_G$, and hence $\ol{\cS}_{(G,G,s)}$, is smooth at $y$.

We  prove that $\ma S_{(G,G,s)}$ is irreducible by showing that so is $\ol{\cS}_{(G,G,s)}$. Let $\ma W_1,\dots,\ma W_p$ be the irreducible components of $\ol{\cS}_{(G,G,s)}$. 
Then $\ma S_{(G',P',s')}=\cup_{i=1}^p(\ma W_i\cap\ma S_{(G',P',s')})$.  
Since the   map $\ma{S}_{(G,G,s)}\ra \cM_G$ is finite and \'etale  each irreducible component of $\ma S_{(G,G,s)}$ dominates $\cM_G$. Hence the restriction of the   map $\ol{\cS}_{(G,G,s)}\ra \ol{\cM}_G$ to each $\ma W_i$ is surjective. Therefore, by (1), every intersection $\ma W_i\cap\ma S_{(G',P',s')}$  is not empty. 
Since $\ma S_{(G',P',s')}$ is irreducible and different $\ma W_i$   cannot intersect   (as $\ol{\cS}_{(G,G,s)}$ is smooth), we get $p=1$, as wanted.
\end{example}

\begin{thm}\label{thm:irr}
$\hcS_{(G, P,s)}$ is irreducible for any stable spin graph $(G,P,s)$. 
\end{thm}

\begin{proof} 
We proceed by induction on $d_G:=\dim\cM_G$. If $d_G=0$, the statement follows from Proposition~\ref{lem:easy} (iii). So we assume $d_G\ge1$.

By \eqref{eq:Sprimedef} we can assume $G=P$. By Proposition \ref{prop:b1} we can assume that $(G,G,s)$ is not basic  and, by  Proposition~\ref{lem:easy} (i) and (ii),  that $G$ has genus $g\ge2$ with $E(G)\ne\emptyset$. Using Lemma \ref{lem:specia}, we fix a stable graph $G'$ contracting to $G$ with $d_{G'}=d_G-1$, and a  spin structure $(P', s')$ on $G'$ such that $\Aut(G',P',s')=\Aut(G')$ and $\gamma_*(P',s')=(G,s)$ for  every contraction $\gamma\col G'\ra G$; recall that   $\gamma$ contracts exactly one edge. By Proposition \ref{thm-strata} and  Lemma~\ref{lem:specia}, we have
\begin{equation}\label{eq:cap}
\ma S_{(G',P',s')}=\ol{\cS}_{(G,G,s)}\cap \pi^{-1}(\cM_{G'}).
\end{equation}

Let $\nu\col \ol{\cM}^\nu_{G}\ra \ol{\cM}_{G}$ be the normalization and, with the notation in Subsection~\ref{rem:defospin}, consider
the following diagram, where the squares are all Cartesian.
\begin{eqnarray*}
\SelectTips{cm}{11}
\begin{xy} <16pt,0pt>:
\xymatrix{
\hcS_{(G,G,s)}\ar@{^(->}[r]\ar[dd]& \ol{\hcS}_{(G,G,s)}\ar[dd]_{\rho}\ar[r] & \ol{\cS}^{\nu}_{(G,G,s)}
\ar @/_1.4pc/[dd]_{\tau}
 \ar@{^(->}[d]\ar[r]& \ol{\cS}_{(G,G,s)} \ar@{^(->}[d]& \\
& & \ol{\ma S}^{\nu}_G  \ar[r]^{\nu^S}\ar[d]& \ol{\ma S}_G \ar@{^(->}[r]\ar[d]  & \Sgnbst\ar[d]_{\pi} \\
\wt{\cM}_G \ar@{^(->}[r] & \ol{\wt{\cM}}_G \ar[r]^{\chi} & \ol{\cM}^{\nu}_G  \ar[r]^{\nu}          & \ol{\cM}_{G}  \ar@{^(->}[r]& \ol{\cM}_{g,n} 
}        
\end{xy}
\end{eqnarray*}
Since $G=P$ we have 
$
\hcS_{(G,G,s)}=\tcS_{(G,G,s)}=\ma S_{(G,G,s)}\times_{\cM_G}\wt{\cM}_G.
$

\emph{Step 1.} 
$\ol{\hcS}_{(G,G,s)}$ is smooth at every point lying over $ \cM_{G'}$.

 Notice that $\ol{\hcS}_{(G,G,s)}\ra \ol{\ma S}^{\nu}_{(G,G,s)}$ is \'etale because so is
 $\chi$ (recall  \eqref{eq:tildebarra}) and $\ol{\cS}^{\nu}_{(G,G,s)}$ is union of irreducible components of $\ol{\cS}^{\nu}_G$.
Hence it suffices to show that
$\ol{\ma S}^{\nu}_G$ is smooth at any point $y^\nu$ lying over  $\cM_{G'} $. We will use Subsection~\ref{rem:defospin}  and  notation  \eqref{eq:Igamma}. Let $x'_\gamma\in \nu^{-1}(\cM_{G'})$ be the point over which $y^\nu$ lies; let $x'=\nu(x'_\gamma)\in \ol{\cM}_{g,n}$ and $y=\nu^S(y^\nu)\in\ol{\ma S}_{g,n}$.
 Locally at $y^\nu$  we have
\[
\wh {\ma O}_{\ol{\ma S}^{\nu}_G,y^\nu}\cong \frac{k[[t_1,\dots,t_{3g-3+n}]]}{\ma I_\gamma}\otimes_{k[[t_1,\dots,t_{3g-3+n}]]} k[[s_1,\dots,s_{3g-3+n}]],
\]
where the vanishing of $t_1$   corresponds to    locally  trivial deformations      at the node  contracted by $\gamma$. Let $\ma I$ be the ideal of $k[[t_1,\dots,t_{3g-3+n}]]$ generated by $t_{2},\dots,t_{3g-3+n}$,
then $\ma I_\gamma\subset \ma I$.  Locally at $x'$ and $y$, the map $\pi$ is induced by   $\pi^\#\col \wh{\ma O}_{\ol{\cM}_{g,n},x'}\ra  \wh{\ma O}_{\Sgnbst,y}$, given $\pi^\#(t_i)=s_i$ for $i\geq 2$, and   by $\pi^\#(t_1)=s^h_1$ with $h=1,2$ depending on, respectively, whether $P'=G'$   or $P'=G'-e'$  
(recall the proof of Lemma \ref{lem:specia}). Therefore $ \ol{\ma S}^{\nu}_G$ is smooth   at $y^\nu$.

\smallskip

\emph{Step 2.} There is a stratum, $\ma N$, of   $\ol{\wt{\cM}}_G$  such that $\rho^{-1}(\ma N)$ is irreducible.

 Consider    $\gamma\col G'\ra G$ and let $e' $ be the   contracted edge. Let $u_1,u_2$ be the vertices incident to $e'$, and $v_0 $ be the vertex of $G$ to which $e'$ is contracted. We introduce a  connected graph,  $H$,  having one edge and two vertices, $u^H_1,u^H_2$, such that, for $1=1,2$, we have
\[
w_H(u^H_i)=w_{G'}(u_i) 
\]
\[
\ell_H(u^H_i)=\deg_{G'}(u_i)+\ell_{G'}(u_i)-1. 
\]
Then $H$ is  a stable graph of genus $g_H=w_H(u^H_1)+w_H(u^H_2)$ with  $n_H=\ell_H(u^H_1)+\ell_H(u^H_2)$ legs. We consider the corresponding codimension-one stratum 
\[
\cM_H\subset \ol{\cM}_{g_H,n_H}=\ol{\cM}_{w_G(v_0),\deg(v_0)+\ell_G(v_0)}
\]
and define 
\[
\ma N:= \cM_H\times(\prod_{v\in V(G)\ssm \{v_0\}}\cM_{w_G(v),\deg_G(v)+\ell_G(v)}).
\] 
Then $\ma N$ can be identified with a codimension-one stratum of  $ \ol{\wt{\cM}}_G$, and the map $\chi\col \ol{\wt{\cM}}_G\ra \ol{\cM}^{\nu}_G$ sends this stratum onto one of the copies, written $\cM^{\gamma}_{G'}$, of $\cM_{G'}$ contained in the pre-image $\nu^{-1}(\cM_{G'})$
(recall from Subsection~\ref{rem:defospin}  that $\nu^{-1}(\cM_{G'})$ is the disjoint union of copies of $\cM_{G'}$, one for each contraction $G'\to G$). 
There is a natural map $\wt{\cM}_{G'}\ra\ma N$ which presents $\ma N$ as a quotient $\ma N=[\wt{\cM}_{G'}/\Aut(H)]$.
Since $\Aut(H)$ is trivial
(as $n_H>0$) we have 
$$\ma N\cong \wt{\cM}_{G'}.$$ 
Moreover, by \eqref{eq:cap}, we have an isomorphism 
$\tau^{-1}(\cM_{G'}^{\gamma})\cong \ma S_{(G',P',s')}$. Hence
\[
\rho^{-1}(\ma N)\cong \ma S_{(G',P',s')}\times_{\cM_{G'}}\wt{\cM}_{G'}\cong \hcS_{(G', P',s')},
\]
where the second isomorphism follows from Lemma~\ref{lm:quotient} (as $\Aut(G',P',s')=\Aut(G')$). 
By induction  $\hcS_{(G',P',s')}$ is irreducible, hence so is  $\rho^{-1}(\ma N)$.

\smallskip

\emph{Step 3.} $\hcS_{(G,G,s)}$ is irreducible.

It suffices to prove that $\ol{\hcS}_{(G,G,s)}$ is irreducible.
By contradiction, assume that $\ol{\hcS}_{(G,G,s)}$ has $p\geq 2$   irreducible components, written
 $\ma W_1,\dots,\ma W_p$. By Step 1, two different components, $\ma W_i$ and $\ma W_j$, do not intersect in $\rho^{-1}(\ma N)$, hence we have a disjoint union 
 $$\rho^{-1}(\ma N)=\sqcup_{i=1}^p \left(\ma W_i\cap \rho^{-1}(\ma N)\right).$$
 
 Now, every irreducible component of $\hcS_{(G,G,s)}$ surjects onto
 $ \widetilde{\cM}_G$, hence every $\ma W_i$ surjects  (via $ \rho$) onto 
 $ \ol{\widetilde \cM}_G$. 
  Therefore  each intersection  $\ma W_i\cap \rho^{-1}(\ma N)$ is not empty. This contradicts   the irreducibility of $\rho^{-1}(\ma N)$.
\end{proof}

The following  is a   consequence of Theorem~\ref{thm:irr} and Proposition~\ref{lem:commute}. 
\begin{thm}\label{cor:irr}
$\ma S_{(G,P,s)}$ is irreducible for any stable spin graph $(G,P,s)$. 
\end{thm}

\subsection{An application}
We here apply our methods  to compute the dimension of the space of sections of theta characteristics on general stable curves. This generalizes (and gives a new proof of)  a  fact well-known for a general smooth curve $X$, namely that 
$h^0(X,L)\leq 1$ for any theta characteristic $L$.

 \begin{prop}\label{h0app}
Let $G$ be a stable graph  and $X$ a general curve whose dual graph is $G$.  
Let  $(\wh X, \wh L)$ be a spin curve  with dual spin graph $(G,P,s)$.
\begin{enumerate}[(a)]
 \item
 \label{h0app1}
 If $|V(G/P)|=1$   then
 $h^0(\wh X, \wh L)\leq 1$.
 \item
 More generally,
 we have
\[
h^0(\wh X, \wh L)=\sum_{v\in V(G/P)} s(v) 
\]  
where the above sum  is over $\ZZ$ (not over $\ZZ/2\ZZ$).
 \end{enumerate}
\end{prop}

\begin{proof}
 We set, as usual,  $R=E\smallsetminus P$ and  $X_R^\nu=\sqcup_{v\in V(G/P)}Z_v$, and we view $Z_v$ as a stable   curve whose dual graph is $\oP_v$.
Our spin curve $(\wh X, \wh L)$ corresponds to a pair $(X_R^\nu,L_R)$, where    $L_R=\wh L|_{X^\nu_R}$ and $L_R^{\otimes 2}\cong \omega_{X^\nu_R}$. We have
\[
h^0(\wh X, \wh L)=h^0(X_R^\nu,L_R)=\sum_{v\in V(G/P)}h^0(Z_v,(L_R)|_{Z_v})  
\]
 (see   Remark \ref{rem:chm}).
Therefore the second part of the statement follows from the first.  We proceed 
  by induction on $d_G:=\dim \ma M_G$. 
  
  If $d_G=0$ the graph  $G$ is  3-regular, and
  one easily checks that  $Z_v$ is either a smooth rational curve, or a cycle of smooth rational curves.  Therefore we have     $h^0(Z_v,(L_R)|_{Z_v})\leq 1$ and both parts of the statement hold. 
Observe that the same reasoning proves the statement for $g\le 1$. 

Now we assume $d_G>0$ 
and  $V(G/P)=\{v\}$. 
Then $Z_v=X_R^\nu$ and  its dual graph is $\oP$.
It suffices to prove  \eqref{h0app1} for the spin graph $(\oP,\oP,s)$.
 
By what we   observed, we can assume $\oP$ has genus $\geq 2$.
If  $\oP$ is a basic graph, then we are done by  Lemma~\ref{lem:thetaform}.

Otherwise, we  shall use a degeneration argument.
If  $E(\oP)\neq \emptyset$, we can  apply Lemma~\ref{lem:specia}, which
gives us a stable spin graph $(G',P',s')$ which contracts to $(\oP,  \oP, s)$.
If $\oP$ has no edges, we can consider a stable spin graph $(G',P',s')$ 
where $G'$ has one vertex and one loop, so that $(G',P',s')$  contracts to $(\oP,  \oP, s)$. In both cases we have     $d_{G'}>d_{\oP}$.
    Let now $X'$ be a general curve with $G'$ as dual graph, and let $(\wh X', \wh L')$ be a spin curve on $X'$ with $(G',P',s')$ as dual spin graph. Since  $|V(G'/P')|=1$ we can use the induction hypothesis, yielding $h^0(\wh X',\wh L')\le 1$.
    
     By Proposition \ref{thm-strata} we have   $ \mathcal S_{(G',P',s')} \subset \overline{\mathcal S}_{(\ol P,\ol P,s)}$, hence the statement follows from the upper semicontinuity of cohomology.
\end{proof}

\section{The tropicalization of $\Sgnbst$}\label{sec:tropSg}

\subsection{The tropicalization map}
Let   $\mathcal Y$ be a proper 
Deligne-Mumford stack over $k$ and let $Y$ be its coarse moduli space. We write
 $Y^{\an}$ for the  Berkovich analytification,  defined in \cite{Berkovich}.
 Since $\ma Y$ is proper, 
for any  non-Archimedean field extension $K$ of the trivially valued field $k$, a $K$-point 
of  $Y^{\an}$  is represented  by a morphism $\Spec R\to \ma Y$, where $R$ is the valuation ring of $K$. 
We denote by $\text{val}_K$ the valuation of $K$ and assume $K$ complete.

We shall consider the analytifications $\Mgnban$ and  $\Sgnban$.
Let $x^{\an}$ be a point in $\Mgnban$. Then, up to field extension, $x^{\an}$ can be represented by a    stable $n$-pointed  curve  $\ma X\ra \Spec R$, where $R$ is  as above. Let $X$ be the reduction over the closed point of $\Spec R$, and  $G$   the dual graph of $X$.  The (well known)   tropicalization map 
\[
\Trop_{\ol{\cM}_{g,n}}\colon \Mgnban\la \ol{M}_{g,n}^{\trop}
\]
takes $x^{\an}$ to the class of the tropical curve $\Gamma=(G,\ell)$, where,
 given a node, $e$, of $X$ and an \'etale neighborhood  where the local equation of $\ma X$ at $e$ is $xy=f_e$, for $f_e\in R$, we have 
$
\ell(e)=\text{val}_K(f_e).
$

Similarly, a point $y^{\an}$ in $\Sgnban$ is represented by a  stable $n$-pointed spin curve  $(\wh{\ma X}, \wh{\ma L})\ra \Spec R$. 
We let $(\hX, \wh{L})$ be the reduction  over $k$ and   write  $X$ for the stable model of $\hX$.
Keeping this notation, we state the following

\begin{defilemma}
\label{defitropS}
We have a   {\emph {tropicalization}} map 
\[
\Trop_{\Sgnbst}\colon \Sgnban\la \ol{S}_{g,n}^{\trop}
\] 
taking $y^{\an}\in \Sgnban$  to the   spin tropical curve $(\Gamma,P,s)$, where $\Gamma=(G,\ell)$ with
 $(G,P,s)$ the dual spin graph of $(\wh X, \wh{L})$, and $\ell$ is defined as follows.
For every  node  $e$   of $X$, let  $e'$ be a node of $\wh X$ lying over   $e$, and let $xy=h_e$ be an \'etale-local equation for $\wh {\ma X}$ at $e'$, with    $h_e\in R$, then 
 $$
\ell(e)={\text{val}}_K(h_e).
$$

\end{defilemma}

\begin{proof}
The map $\Trop_{\Sgnbst}$ is obviously well defined if 
 there is only one node, $e'$, of $\hX$ lying over the node $e$ of $X$.
 Suppose  we have two  nodes, $e',e''$    lying over $e$. By the description of the universal deformation of a spin curve given in \cite[Sect. 3.2]{CCC07} (in particular  Eq. (4)), there are \'etale neighborhoods of $\widehat {\ma X}$ around $e'$ and $e''$ in which the local equations of $\widehat{\ma X}$ are respectively $xy=h_e$ and $x'y'=h_e$, for the same  $h_e\in R$.
Hence 
$\Trop_{\Sgnbst}$ is well defined.
One can show as in  \cite[Lem-Def. 2.4.1.]{Viv13} that $\Trop_{\Sgnbst}$ is independent of the choice of $R$ and of the local equation.  
\end{proof}

The   map $\pi\colon \Sgnbst\ra \ol{\cM}_{g,n}$ induces a   map of Berkovich analytic spaces,
$$
\pi^{\an}\col\Sgnban \longrightarrow \Mgnban,
$$
defined as follows. Let   $y^{\an}\in \Sgnban$  be represented by    $(\wh{\ma X}, \wh{\ma L})\to \Spec R$,
and    let $\psi\col \Spec R\to \Sgnbst$ be the corresponding map.
 Then $x^{\an}:=\pi^{\an}(y^{\an})$ 
 is  represented by the morphism $\pi\circ\psi\col \Spec R\ra \ol{\cM}_{g,n}$.
 It is clear that
  $x^{\an}$ is represented by  the stable model,   $\cX\to \Spec R$, of  $\wh{\ma X}\to \Spec R$.

We shall prove in Theorem~\ref{thm:skeleton} that the tropicalization maps are compatible with $\pi^{\trop}$ and $\pi^{\an}$ i.e. 
 that
 $\pi^{\trop}\circ \Trop_{\Sgnbst}= \Trop_{\Mgnbst}\circ \pi^{\an}$.

 \subsection{The skeleton of $\Sgnbst$}

We recall some results for toroidal embeddings of Deligne-Mumford stacks from  \cite[Sect. 6]{ACP15}.

Let $\mathcal Y$ be  as above and assume that $U\subset \ma Y$ is a toroidal embedding of Deligne-Mumford stacks. Given a scheme $V$ and an \'etale morphism $V\ra \ma Y$, one considers the group, $\MON_V$, of Cartier divisors on $V$ supported on $V\smallsetminus U_V$, where $U_V=U\times_{\ma Y} V$, and  the submonoid  $\EFF_V\subset \MON_V$ of effective divisors. Let $\MON_{\ma Y}$ and $\EFF_{\ma Y}$ be the \'etale sheaves on $\ma Y$ associated to these presheaves.  

Given a stratum $W\subset \ma Y$ and a geometric point $y$ of $W$, there is a natural action of the \'etale fundamental group $\pi_1^{et}(W,y)$ on the stalk $\MON_{\ma Y,y}$ that preserves the stalk $\EFF_{\ma Y,y}\subset \MON_{\ma Y,y}$. The \emph{monodromy group} $H_W$ is the image of $\pi_1^{et}(W,y)$ in $\text{Aut}(\MON_{\ma Y,y})$.
We define the extended cone
\begin{equation}\label{E:coneW}
\ol{\sigma}_W:=\Hom (\EFF_{\ma Y,y},\ol{\mathbb R}_{\geq0}).
\end{equation}
The \emph{skeleton} $\ol{\Sigma}(\ma Y)$ of $\ma Y$ is the extended generalized cone complex
\begin{equation}\label{E:GenCone}
\ol{\Sigma}(\ma Y):=\lim_{\longrightarrow}\ol{\sigma}_W,
\end{equation}
where $\ol{\sigma}_W\ha\ol{\sigma}_{W'}$ for $W'\subset \ol{W}$, 
and the automorphisms of $\ov{\sigma}_W$ are induced by the elements of the monodromy group $H_W$.  

There is a remarkable retraction,   ${\bf p}_{\mathcal Y}: Y^{\an}\ra \ol{\Sigma}(\mathcal Y)$, described as follows. Assume that the point $y^{\an}$ of $Y^{\an}$ is represented by a map $\psi\col \Spec R\to \ma Y$. Let $y\in \ma Y$ be the image of the closed point of $\Spec R$ and let $W$ be the stratum of $\ma Y$ containing $y$. We have a chain of maps
\begin{equation}
 \label{chain}
\xymatrix@=.4pc{
\EFF_{\ma Y,y}    \ar[rrr]^{\epsilon}   &&& \widehat{\ma O}_{\ma Y,y}\ar[rrr]^{\psi^\#}  &&& R\ar[rrrr]^{\val_K}  &&&& \ol{\mathbb R}_{\geq0}   
}
\end{equation}
where $\epsilon$ is the map that takes an effective divisor to its local equation. 
Then ${\bf p}_{\ma Y}(y^{\an})\in \ol{\Sigma}(\ma Y)$ is the equivalence class of the  homomorphism  
\eqref{chain},
 which is an element of $\ol{\sigma}_{W}$. 

The inclusion $\cM_{g,n}\subset \ol{\cM}_{g,n} $ is a toroidal embedding of Deligne-Mumford stacks (see \cite[Sect.  3.3]{ACP15}),
so we can consider the corresponding skeleton   $\ol{\Sigma}(\ol{\ma{M}}_{g,n})$. As for $\Sgnbst$, we have a similar situation.

\begin{prop}\label{lem:tormap}
The inclusion $\ma S_{g,n}\subset \Sgnbst$ is a toroidal embedding of Deligne-Mumford stacks. 

The map $\pi\colon \Sgnbst\ra \ol{\cM}_{g,n}$ is a toroidal morphism of (toroidal) Deligne-Mumford stacks.

The map $\pi^{\an}\col \Sgnban\ra \Mgnban$ restricts to a map of extended generalized cone complexes $\ol{\Sigma}(\pi)\col  \ol{\Sigma}(\ol{\ma{S}}_{g,n}) \la  \ol{\Sigma}(\ol{\ma{M}}_{g,n})$. 
\end{prop}

\begin{proof}
Let $x$ and $y$ be closed points of $\ol{\cM}_{g,n}$ and $\Sgnbst$ such that  $x=\pi(y)$. 
Let $S_{(G,P,s)}$ be the stratum containing $y$.
With  the notation of  Subsection~\ref{rem:defospin}, we write $E=\{e_1,\ldots, e_{\delta}\}$
so that $P=\{e_{r+1},\ldots, e_{\delta}\}$. 
 
Locally at $y$, the boundary divisor $\Sgnbst\smallsetminus \ma S_{g,n}$ is given by $\prod_{i=1}^\delta s_i=0$, so $\ma S_{g,n}\subset \Sgnbst$ is a toroidal embedding of Deligne-Mumford stacks. 

Next, $\pi$ is toroidal because, locally at $x$ and $y$, it is induced by the ring homomorphism $\pi^\#\col\wh {\ma O}_{\ol{\cM}_{g,n},x}\ra \wh {\ma O}_{\Sgnbst,y}$ given by $\pi^\#(t_i)=s_i^2$ for $i\le r$, and $\pi^\#(t_i)=s_i$ for $i>r$. 
The last part follows  from  \cite[Prop. 6.1.8]{ACP15} 
 \end{proof}

We denote by $\Sigma(\Mgnbst)$, respectively 
$\Sigma(\Sgnbst)$,  the generalized cone complex obtained by restricting $\ol{\Sigma}(\Mgnbst)$, resp.   $\ol{\Sigma}(\Sgnbst)$,
to its ``finite part''.

By \cite[Thm. 1.2.1]{ACP15}, there is an isomorphism of extended generalized cone complexes 
$
\ol{\Phi}_{\ol{\cM}_{g,n}}\col \ol{\Sigma}(\ol{\ma{M}}_{g,n})\la \ol{M}_{g,n}^{\trop}
$
which restricts to an isomorphism
$\Sigma(\Mgnbst)\cong \Mgnt$, and such that 
 $\Trop_{\Mgnbst}=\ol{\Phi}_{\ol{\cM}_{g,n}}\circ {\bf p}_{\Mgnbst}$.

\begin{thm}\label{thm:skeleton}
 There is an isomorphism of extended generalized cone complexes, $\ol{\Phi}_{\Sgnbst}\col  \ol{\Sigma}(\ol{\ma {S}}_{g,n}) \la  \ol{S}_{g,n}^{\trop}$, 
which restricts to an isomorphism $\Sigma(\ol{\ma {S}}_{g,n})\cong   S_{g,n}^{\trop}$.
Moreover, the following diagram is  commutative 
\begin{eqnarray*}
\SelectTips{cm}{11}
\begin{xy} <16pt,0pt>:
\xymatrix{
\Sgnban
\ar[d]_{\pi^{\an}} \ar[rr]^{{\bf p}_{\ol{\ma {S}}_{g,n}}\;}   
  && \ar[rr]^{\ol{\Phi}_{\Sgnbst}} \ar[d]_{\ol{\Sigma}(\pi)} \ol{\Sigma}(\ol{\ma{S}}_{g,n}) && \ol{S}_{g,n}^{\trop}   \ar[d]_{\pi^{\trop}} \\
\Mgnban  
\ar[rr]^{{\bf p}_{\ol{\ma {M}}_{g,n}}\;}            &&   \ar[rr]^{\ol{\Phi}_{\ol{\cM}_{g,n}}}         \ol{\Sigma}(\ol{\ma{M}}_{g,n})  && \ol{M}_{g,n}^{\trop}  
 }
\end{xy}
\end{eqnarray*}
and  $\Trop_{\Sgnbst}=\ol{\Phi}_{\Sgnbst}\circ{\bf p}_{\ol{\ma S}_{g,n}}$.
\end{thm}

\begin{proof}
 Recall our definition,
$
\ol{S}_{g,n}^{\trop}=\underset{\longrightarrow}{\text{lim}}\ \ol{\sigma}_{(G,P,s)}
$, 
where the right side is the colimit of the diagram of  extended cones $\ol{\sigma}_{(G,P,s)}$, with $\ol{\sigma}_{(H,Q,s')}\ha  \ol{\sigma}_{(G,P,s)}$ for $(G,P,s)\ge(H,Q,s')$, and the automorphisms of   $\ol{\sigma}_{(G,P,s)}$ are induced by   $\text{Aut}(G,P,s)$. 

Consider a stratum $W=\ma S_{(G,P,s)}$ of $\Sgnbst$. Given a point $y\in W$, by Subsection~\ref{rem:defospin} we have an isomorphism of monoids $\EFF_{\Sgnbst,y}\ra \mathbb Z_{\ge 0}^{E}$, so by \eqref{E:coneW}, the extended cone  $\ol{\sigma}_W$  is naturally isomorphic to $\ol{\sigma}_{(G,P,s)}$. 
We can thus identify them and  rewrite  \eqref{E:GenCone} as follows
\[
\ol{\Sigma}(\ol{\ma {S}}_{g,n})=\underset{\longrightarrow}{\text{lim}}\ \ol{\sigma}_{(G,P,s)},
\]
where $\ol{\sigma}_{(H,Q,s')}\ha \ol{\sigma}_{(G,P,s)}$ for $\ol{\ma S}_{(G,P,s)}\subset \ol{\cS}_{(H,Q,s')}$, and the automorphisms of  $\ol{\sigma}_{(G,P,s)}$ are induced by the elements of the monodromy group $H_{\ma S_{(G,P,s)}}$. 

By  Proposition \ref{thm-strata}, the existence of an isomorphism $\ol{\Phi}_{\Sgnbst}$ as in the statement follows once we prove that $H_{\ma S_{(G,P,s)}}=\Aut(G,P,s)$ for every   $(G,P,s)$. 
 To show this, let $y\in \ma S_{(G,P,s)}$ and recall   diagram \eqref{diag:sprime}. 
By subsection~\ref{rem:defospin}, the  set $E$ determines  a group basis for $\MON_{\Sgnbst,y}$ and a monoid basis for $\EFF_{\Sgnbst,y}$. The locally constant sheaf of sets on $\ma S_{(G,P,s)}$  whose stalk at every point is   
the set of nodes of the underlying stable curve
becomes trivial when pulled back to  $\hcS_{(G,P,s)}$. Hence the pull-backs of $\MON_{\Sgnbst}$ and $\EFF_{\Sgnbst}$ to  $\hcS_{(G,P,s)}$ are trivial. By Proposition~\ref{lem:commute} and Theorem~\ref{thm:irr},  the action of $\pi_1^{et}(\ma S_{(G,P,s)},y)$ on $\MON_{\Sgnbst,y}$ factors through its quotient $\Aut(G,P,s)$, and hence  $H_{\ma S_{(G,P,s)}}=\Aut(G,P,s)$.

We now prove  that $\Trop_{\Sgnbst}=\ol{\Phi}_{\Sgnbst}\circ{\bf p}_{\ol{\ma S}_{g,n}}$ and that the diagram   in the statement is commutative; its left square is so by
  \cite{ACP15}.

 Consider a point $y^{\an}$ in $ \Sgnban$ given by  $\psi\col \Spec R\to \Sgnbst$. Assume that the image, $y$, of the closed point of $\Spec R$ lies in the stratum $\ma S_{(G,P,s)}$. 
Let $\Trop_{\Sgnbst}(y^{\an})=[(\Gamma,P,s)]\in \ol{S}_{g,n}^{\trop}$, where $\Gamma=(G,\ell)$. The set $E$ can be seen as a  basis of the free monoid $\EFF_{\Sgnbst,y}$, and $\ell$ is given by the composition
$$
\xymatrix@=.4pc{
\ell:& E    \ar[rrr]   &&& \EFF_{\Sgnbst,y}\ar[rrrrr]^{{\bf p}_{\Sgnbst}(y^{\an})}  &&&&& \ol{\mathbb R}_{\geq0}   
}
$$

 Recalling the definition of  ${\bf p}_{\ol {\cS}_{g,n}}(y^{\an})$ given in \eqref{chain}, we obtain   $$\Trop_{\Sgnbst}(y^{\an})=\ol{\Phi}_{\Sgnbst}\circ {\bf p}_{\ol {\cS}_{g,n}}(y^{\an}).$$

Let $x^{\an}=\pi^{\an}(y^{\an})$, then 
 $x^{\an}$ is represented  by  $\pi\circ\psi\col \Spec R\ra \ol{\cM}_{g,n}$. With the notations of Subsection~\ref{rem:defospin},  $\Trop_{\ol{\cM}_{g,n}}(x^{\an})$ is the tropical curve  
 $(G,\ell)$
    such that
 
    \[
\ell(e_i)=\val _K(\psi^\#(\pi^\#(t_i)))=
\begin{cases}
\begin{array}{ll}
2    \val _K(\psi^\#(s_i) )  & \text{if } 1\leq i\leq r \\
    \val _K(\psi^\#(s_i) )   & r<i\leq \delta 
\end{array}
\end{cases}
\]
since we have $\pi^\#(t_i)=s_i^2$ for $i\le r$ and $\pi^\#(t_i)=s_i$ for $i>r$.
    
    On the other hand,  $\Trop_{\Sgnbst}(y^{\an})$  is the tropical spin curve in  $ \ol{S}_{(G,P,s)}^{\trop}$ 
    such that the edge $e_i$ of $G$ has length 
    $ 
    \val _K(\psi^\#(s_i) ).
    $ 
  By the definition of $\pi^{\trop}$, 
  the length of $e_i$ on the curve $\pi^{\trop}\circ\Trop_{\Sgnbst}(y^{\an})$
  is    $2 \val _K(\psi^\#(s_i) )$ if $i\leq r$ and $ \val _K(\psi^\#(s_i) )$ otherwise; so it is equal to $\ell(e_i)$.
  We thus proved $\pi^{\trop}\circ\Trop_{\Sgnbst}(y^{\an})=\Trop_{\ol{\cM}_{g,n}}\circ\pi^{\an}(y^{\an})$.
\end{proof}

\noindent
{\it{Acknowledgements.}}
We  thank Alex Abreu, Eduardo Esteves, Martin Ulirsch, and Filippo Viviani for several useful remarks. 
Part of the material in this paper is based upon work supported by the National Science Foundation under Grant No. DMS-1440140 while the first named author was visiting the Mathematical Sciences Research Institute in Berkeley, California.

\end{document}